\documentclass[review,hidelinks,onefignum,onetabnum]{siamart250211}

\usepackage{lipsum}
\usepackage{amsfonts}
\usepackage{graphicx}
\usepackage{epstopdf}
\usepackage{algorithmic}
\ifpdf
  \DeclareGraphicsExtensions{.eps,.pdf,.png,.jpg}
\else
  \DeclareGraphicsExtensions{.eps}
\fi

\newsiamremark{remark}{Remark}
\newsiamremark{hypothesis}{Hypothesis}
\crefname{hypothesis}{Hypothesis}{Hypotheses}
\newsiamthm{claim}{Claim}
\newsiamremark{fact}{Fact}
\crefname{fact}{Fact}{Facts}

\headers{One-step schemes for parabolic type equations}{Xiaoyi Li, Aijie Cheng, and Zhengguang Liu}

\title{A new class of one-step A-stable and L-stable schemes of high-order accuracy for parabolic type equations\thanks{Submitted to the editors DATE.\funding{This work was supported by the Natural Science Foundation of Shandong Province (Grant Nos: ZR2023YQ007 and ZR2024MA077), the Taishan Scholars Program of Shandong Province of China (Grant No: tsqn202408140).}}}

\author{Xiaoyi Li\thanks{School of Mathematics, Shandong University, Jinan, 250100, Shandong, China(\email{202311818@mail.sdu.edu.cn}).}
\and Aijie Cheng\thanks{Corresponding author. School of Mathematics, Shandong University, Jinan, 250100, Shandong, China (\email{aijie@sdu.edu.cn}).}
\and Zhengguang Liu\thanks{Corresponding author. School of Mathematics and Statistics, Shandong Normal University, Jinan, 250358, Shandong, China(\email{liuzhg@sdnu.edu.cn}).}}

\usepackage{amsopn}
\DeclareMathOperator{\diag}{diag}

\ifpdf
\hypersetup{
  pdftitle={A new class of one-step A-stable and L-stable schemes of high-order accuracy for parabolic type equations},
  pdfauthor={Xiaoyi Li, Aijie Cheng, and Zhengguang Liu}
}
\fi

\begin{document}

\maketitle

\begin{abstract}
Recently, a new class of BDF schemes proposed in [F. Huang and J. Shen, SIAM  J Numer. Anal., 62.4, 1609--1637] for the parabolic type equations are studied in this paper. The basic idea is based on the Taylor expansions at time $t^{n+\beta}$ with $\beta>1$ being a tunable parameter. These new BDF schemes allow larger time steps at higher order r for stiff problems than that which allowed with a usual higher-order scheme. However, multi-step methods like BDF exhibit inherent disadvantages relative to one-step methods in practical implementations.  In this paper, inspired by their excellent work, we construct a new class of high-order one-step schemes for linear parabolic-type equations. These new schemes, with several suitable $\beta_i$, can achieve A-stable, or even L-stable. Specially, the new scheme with special parameters $\beta_i$ can be regarded as the classical one-step Runge-Kutta scheme with a stabilized term. Besides, we provide two different techniques to construct the one-step high-order schemes: the first one is by choosing different parameters $\beta_i$, and the second one is by increasing the number of intermediate layers. Both methods have been proven to be highly effective and even exhibit superconvergence property. Finally, we also conducted several numerical experiments to support our conclusions.
\end{abstract}



\section{Introduction}
Parabolic partial differential equations (PDEs) provide fundamental mathematical frameworks for modeling critical physical phenomena. Key applications include thermal and mass diffusion processes \cite{baer1998heat}, viscous fluid dynamics \cite{gassner2008discontinuous,GASSNER20071049}, and resistive magnetohydrodynamics \cite{dumbser2009high}. Diverse subscale physical systems further involve nonlinear parabolic or coupled hyperbolic–parabolic PDEs, exemplified by detonation shock dynamics \cite{aslam1996level}, flame propagation \cite{markstein2014nonsteady}, Stefan problems \cite{rubinshteuin1971stefan}, and anisotropic thermal conduction \cite{balsara2008simulating}. Parabolic formulations also underpin mathematical techniques in level set methods, image processing, and computational geometry \cite{osher2004level,sethian1999level}, demonstrating broad interdisciplinary relevance.
 
One-step methods advance solutions directly from $t^n\rightarrow t^{n+1}$ in a single computational step, circumventing the storage overhead and startup complexities inherent in multi-step formulations. For high-order one-step schemes, we typically extend them based on the Runge-Kutta (RK) methods which constitute a cornerstone of one step numerical integration for ordinary differential equations (ODEs) and semi-discretized partial differential equations (PDEs). Renowned for their high-order accuracy, self-starting capability, and flexibility in handling stiffness, RK methods have evolved from their classical 4th-order formulations to sophisticated high-order ($q\geq5$) and adaptive implementations. Constructing exponential Runge-Kutta methods of collocation type and analyzing their convergence properties for linear and semilinear parabolic problems in \cite{hochbruck2005exponential}. The high-order MBP-preserving time integration scheme is studied using the integrating factor Runge-Kutta (IFRK) method \cite{ju2021maximum}. Extrapolated Stabilized Explicit Runge-Kutta methods (ESERK) are proposed  in \cite{martin2016extrapolated} to solve multi-dimensional nonlinear partial differential equations. Building temporally first- and second-order accurate super-time-stepping methods in \cite{meyer2014stabilized} around the recursion relation associated with Legendre polynomials.

In \cite{huang2024new}, Huang and Shen developed a series of high order implicit-explicit schemes by considering generalized backward differentiation formulas (GBDF) which incorporating an adaptive parameter $\beta$ that enables judicious balancing of stability domain extension and truncation error control. The GBDF methodology employs temporal Taylor expansions centered at the shifted time node $t^{n+\beta}$, where the adaptive control parameter $\beta\geq1$ governs scheme characteristics. It is well known that the BDF schemes are typical multi-step methods. Multi-step methods like BDF exhibit inherent disadvantages relative to one-step methods in practical implementations. Firstly, they are not self-starting, requiring auxiliary methods (typically one-step schemes) to compute the necessary initial steps, increasing complexity. Secondly, adaptive step-size control presents significant challenges; changing the step-size necessitates recalculating or interpolating historical solution values to maintain consistency, introducing computational overhead and potential stability errors. Thirdly, memory requirements scale with the step number $(k)$, posing storage bottlenecks for large-scale systems—unlike low-storage one-step methods. Additionally, achieving high-order convergence while retaining strong stability (e.g., A-stability) is fundamentally constrained. Finally, multi-step methods are more sensitive to solution discontinuities due to their dependence on historical data continuity, whereas one-step methods rely only on local state information, enhancing robustness for non-smooth problems.

The core idea of this paper is to combine several GBDF schemes with different $\beta$ to construct a new class of high-order one-step schemes. We consider in this paper numerical approximations of a general linear ordinary or partial differential equations given by
\begin{equation}\label{eq1}
\begin{aligned}
&u_t+\mathcal Lu(\mathbf{x},t)=f(\mathbf{x},t),\;0<t<T,\ \mathbf{x}\in\Omega. \\
&u(0)=u_0,
\end{aligned}
\end{equation}
where $\mathcal L$ is a linear operator, which exact descriptions can be found in the next section.

The key point of the one-step method we proposed is to introduce multiple intermediate layers $u^{n+\frac1k},\;\cdots,\;u^{n+\frac{k-1}{k}}$ between $u^n$ and $u^{n+1}$, and then use Taylor expansion to derive an expression for $u^{n+\beta}$ and $u_t^{n+\beta}$. It is obvious that for each $\beta$ provided, an equation will be obtained. This means that if we need k different $\beta$ values, we can obtain a system of equations that can be solved for $u^{n+1}$.

The main contributions of this work are as follows:
\begin{itemize}

	\item we propose a series of one-step schemes of arbitrary high order accuracy for parabolic-type equations;
	
	\item these one-step schemes are A-stable or even L-stable by choosing appropriate parameters which enjoy much more robust and can use much larger
time steps;

	\item we provide two different techniques to construct the one-step high-order schemes: the first one is by choosing different parameters $\beta_i$, and the second one is by increasing the number of intermediate layers.

\end{itemize}
Furthermore, we provide convincing numerical evidence to validate our theoretical findings.

The rest of this paper is organized as follows. In Section 2, we describe the abstract setting and construct the GBDF schemes based on the Taylor expansion at time $t^{n+\beta}$ and briefly recalled the Runge-Kutta scheme. In Section 3, we introduce two different techniques to construct high-order one-step schemes. In Section 4, we investigate their stability regions and give detailed conditions to construct A-stable and L-stable one-step schemes. In Section 5, we provide numerical examples to show the advantages of the new schemes, followed by some concluding remarks in Section 6.

\section{A brief review of GBDF scheme and Runge-Kutta scheme}
The GBDF and Runge-Kutta methods are two typical high-order discretization methods. In order to motivate our new schemes, we briefly review below the GBDF and Runge-Kutta methods. To establish rigorous mathematical foundations, we first formalize the discrete temporal framework: we divide interval [0,T] into N equal parts, with each small interval having a length of $\Delta t=\frac TN$. Subsequently, we proceed to define $t^n=n\Delta t$. The function space is endowed with the canonical $L^2(\Omega)$ inner product $(u,v):=\int_\Omega u(x)v(x)\text dx$ and its induced energy norm $\Vert u\Vert=(u,u)^{\frac12}$. Furthermore we introduce the operator-dependent seminorm $\vert u\vert:=\Vert\mathcal L^{\frac12}u\Vert=(\mathcal Lu,u)^{\frac12}$ to characterize solutions' regularity in the energy space. Building upon this notational scaffolding, we concisely revist the GBDF schemes.

\subsection{The GBDF scheme}

Given an integer $k\geq2$, it follows from the Taylor expansion at time $t^{n+\beta}$ that
\begin{equation}
u(\mathbf{x},t^{n+1-i})=\sum_{j=0}^{k-1}[(1-i-\beta)\Delta t]^j\frac{u^{(j)}(\mathbf{x},t^{n+\beta})}{j!}+\mathcal O(\Delta t^k),\,\forall i\leq k-1,
\end{equation}
where $u^{(j)}(\mathbf{x},t^{n+\beta})=\frac{\partial^j}{\partial t^j}u(\mathbf{x},t^{n+\beta})$.

Then we can derive from the above an implicit difference formula to approximate $\partial_tu(\mathbf{x},t^{n+\beta})$
\begin{equation}
\label{2.2}
\frac{1}{\Delta t}\sum_{q=0}^{k}a_{k,q}(\beta)u(\mathbf{x},t^{n+1-k+q})=\partial_tu(\mathbf{x},t^{n+\beta})+\mathcal O(\Delta t^k),
\end{equation}
where $a_{k,q}(\beta)$ can be uniquely determined by solving the following linear system with a Vandermonde matrix:
\begin{equation}
\begin{bmatrix}
1 & 1 & \cdots & 1 \\
\beta-1 & \beta & \cdots & \beta+k-1 \\
(\beta-1)^2 & \beta^2 & \cdots & (\beta+k-1)^2 \\
\vdots & \vdots & \ddots & \vdots \\
(\beta-1)^{k} & \beta^{k} & \cdots & (\beta+k-1)^{k}
\end{bmatrix}
\begin{bmatrix}
a_{k,k}(\beta) \\
a_{k,k-1}(\beta) \\
a_{k,k-2}(\beta) \\
\vdots \\
a_{k,0}(\beta) \\
\end{bmatrix}
=
\begin{bmatrix}
0 \\
-1 \\
0 \\
\vdots \\
0
\end{bmatrix}
.
\end{equation}

Similarly, we can derive an implicit difference formula to approximate $u(\mathbf{x},t^{n+\beta})$,
\begin{equation}
\label{2.3}
\sum_{q=0}^{k-1}b_{k,q}(\beta)u(\mathbf{x},t^{n+2-k+q})=u(\mathbf{x},t^{n+\beta})+\mathcal O(\Delta t^k),
\end{equation}
with $b_{k,q}$ being the unique solution of the following Vandermonde system:
\begin{equation}
\nonumber
\begin{bmatrix}
1 & 1 & \cdots & 1 \\
\beta-1 & \beta & \cdots & \beta+k-2 \\
(\beta-1)^2 & \beta^2 & \cdots & (\beta+k-2)^2 \\
\vdots & \vdots & \ddots & \vdots \\
(\beta-1)^{k-1} & \beta^{k-1} & \cdots & (\beta+k-2)^{k-1}
\end{bmatrix}
\begin{bmatrix}
b_{k,k-1}(\beta) \\
b_{k,k-2}(\beta) \\
b_{k,k-3}(\beta) \\
\vdots \\
b_{k,0}(\beta) \\
\end{bmatrix}
=
\begin{bmatrix}
1 \\
0 \\
0 \\
\vdots \\
0
\end{bmatrix}
.
\end{equation}

Then, based on (\ref{2.2}) and (\ref{2.3}), the GBDF scheme for (\ref{eq1}) is
\begin{equation}
\label{2.4}
\frac{1}{\Delta t}\sum_{q=0}^{k}a_{k,q}(\beta)u^{n+1-k+q}+\mathcal L\left[\sum_{q=0}^{k-1}b_{k,q}(\beta)u^{n+2-k+q}\right]=f(\mathbf{x},t^{n+\beta}),\,\,k\geq2.
\end{equation}

For the reader's convenience, we list below the coefficients in (\ref{2.4}) for $k=2,3,4$.\\

$k=2$:

\begin{equation}
\nonumber
\begin{aligned}
&a_{2,2}(\beta)=\frac{2\beta+1}{2},\;a_{2,1}(\beta)=-2\beta,\;a_{2,0}(\beta)=\frac{2\beta-1}{2}, \\
&b_{2,1}(\beta)=\beta,\;b_{2,0}(\beta)=-(\beta-1).
\end{aligned}
\end{equation}

$k=3$:

\begin{equation}
\nonumber
\begin{aligned}
&a_{3,3}(\beta)=\frac{3\beta^2+6\beta+2}{6},\;a_{3,2}(\beta)=-\frac{9\beta^2+12\beta-3}{6},\\
&a_{3,1}(\beta)=\frac{9\beta^2+6\beta-6}{6},\;a_{3,0}(\beta)=-\frac{3\beta^2-1}{6},\\
&b_{3,2}(\beta)=\frac{\beta^2+\beta}{2},\;b_{3,1}(\beta)=-(\beta^2-1),\;b_{3,0}(\beta)=\frac{\beta^2-\beta}{2}.
\end{aligned}
\end{equation}

$k=4$:

\begin{equation}
\nonumber
\begin{aligned}
&a_{4,4}(\beta)=\frac{2\beta^3+9\beta^2+11\beta+3}{12},\;a_{4,3}(\beta)=-\frac{8\beta^3+30\beta^2+20\beta-10}{12},\\
&a_{4,2}(\beta)=\frac{12\beta^3+36\beta^2+6\beta-18}{12},\;a_{4,1}(\beta)=-\frac{8\beta^3+18\beta^2-4\beta-6}{12},\\
&a_{4,0}(\beta)=\frac{2\beta^3+3\beta^2-\beta-1}{12},\\
&b_{4,3}(\beta)=\frac{\beta^3+3\beta^2+2\beta}{6},\;b_{4,2}(\beta)=-\frac{\beta^3+2\beta^2-\beta-2}{2},\\
&b_{4,1}(\beta)=\frac{\beta^3+\beta^2-2\beta}{2},\;b_{4,0}(\beta)=-\frac{\beta^3-\beta}{6}.
\end{aligned}
\end{equation}

\begin{remark} 
When $\beta=1$, equation (\ref{2.4}) becomes the classical semi-implicit BDF scheme, and there have been extensive works regarding its stability and error analysis \cite{akrivis2016backward,akrivis2015fully,li2020long} in the literature.
\end{remark}

\subsection{Runge-Kutta scheme}

Let $s$ be an integer (the "number of stages") and $c_{2,1},\;c_{3,1},\;c_{3,2},\;\cdots,\;c_{s,1},\;\cdots,\;c_{s,s-1}$, $d_1,\;\cdots,\;d_s$ and $e_2,\;\cdots,\;e_s$ be real coefficients. Then the method 
\begin{equation}
\label{RK_scheme}
\begin{aligned}
k_1&=f(\mathbf{x},t^n)-\mathcal Lu^n, \\
k_2&=f(\mathbf{x},t^n+e_2\Delta t)-\mathcal L(u^n+\Delta tc_{2,1}k_1), \\
k_3&=f(\mathbf{x},t^n+e_3\Delta t)-\mathcal L(u^n+\Delta t(c_{3,1}k_1+c_{3,2}k_2)), \\
&\cdots \\
k_s&=f(\mathbf{x},t^n+e_s\Delta t)-\mathcal L(u^n+\Delta t(c_{s,1}k_1+\cdots+c_{s,s-1}k_{s-1})), \\
u^{n+1}&=u^n+\Delta t(d_1k_1+\cdots+d_sk_s)
\end{aligned}
\end{equation}
is called an $s$-stage explicit Runge-Kutta method (ERK) for system (\ref{eq1}).

Usually, the $e_i$ satisfies the following conditions
\begin{equation}
e_i=\sum_{j=1}^{i-1}c_{i,j}.
\end{equation}

By adding some constraints on the parameter $c,\;d$ and $e$, we can obtain the commonly seen Runge-Kutta schemes. For instance, the second-order Runge-Kutta method needs to satisfy:
\begin{equation}
d_1+d_2=1,\;d_2e_2=\frac12,
\end{equation}
and the third-order Runge-Kutta method also needs to satisfy:
\begin{equation}
\begin{aligned}
&d_1+d_2+d_3=1,\;d_2e_2+d_3e_3=\frac12,\\
&d_2e_2^2+d_3e_3^2=\frac13,\;c_{3,2}d_3e_2=\frac16.
\end{aligned}
\end{equation}

An extensive introduction, investigation and discussion of numerical methods for ODEs can be found in the text books by Hairer, Norsett and Wanner \cite{hairer1993sp} and Hairer and Wanner \cite{wanner1996solving}.

\section{The new one-step schemes}
In this section, we will provide two different techniques to construct the one-step high-order schemes: the first one is by choosing different parameters $\beta_i$, and the second one is by increasing the number of intermediate layers. We can either use them individually or combine them to construct a series of high-order one-step schemes.

\subsection{The first high-order one-step scheme}
Without loss of generality, we use $u(t)$ to denote $u(\mathbf{x},t)$ and consider to construct $k$-th order ($2\leq k\leq4$) one-step schemes by introducing one intermediate layer and different $\beta_1$ and $\beta_2$. We shall first construct a class of new one-step schemes for system (\ref{eq1}) based on the Taylor expansion at time $t^{n+\beta}$. We can effortlessly obtain the taylor expansions of $u(t^{n+1})$ at time $t^{n+\beta}$, that is
\begin{equation}
\label{taylor}
\begin{aligned}
u(t^{n+1})=&u(t^{n+\beta})-u_t(t^{n+\beta})(\beta-1)\Delta t+u^{(2)}(t^{n+\beta})\frac{(\beta-1)^2\Delta t^2}{2} \\
&-u^{(3)}(t^{n+\beta})\frac{(\beta-1)^3\Delta t^3}{6}+u^{(4)}(t^{n+\beta})\frac{(\beta-1)^4\Delta t^4}{24}+C\Delta t^5.
\end{aligned}
\end{equation}
Similarly, we can also obtain the following taylor expansions of $u(t^{n}),\;u(t^{n+\frac12})$ at time $t^{n+\beta}$, those are
\begin{equation}
\begin{aligned}
u(t^n)=&u(t^{n+\beta})-u_t(t^{n+\beta})\beta\Delta t+u^{(2)}(t^{n+\beta})\frac{\beta^2\Delta t^2}{2} \\
&-u^{(3)}(t^{n+\beta})\frac{\beta^3\Delta t^3}{6}+u^{(4)}(t^{n+\beta})\frac{\beta^4\Delta t^4}{24}+C\Delta t^5, \\
u(t^{n+\frac12})=&u(t^{n+\beta})-u_t(t^{n+\beta})(\beta-\frac12)\Delta t+u^{(2)}(t^{n+\beta})\frac{(\beta-\frac12)^2\Delta t^2}{2} \\
&-u^{(3)}(t^{n+\beta})\frac{(\beta-\frac12)^3\Delta t^3}{6}+u^{(4)}(t^{n+\beta})\frac{(\beta-\frac12)^4\Delta t^4}{24}+C\Delta t^5.
\end{aligned}
\end{equation}
Subsequently, we can identify a set of coefficients $A_{2,1}(\beta)$, $A_{2,2}(\beta)$, and $A_{2,3}(\beta)$ such that the following equation holds true.
\begin{equation}
\label{eq2}
\begin{aligned}
&A_{2,1}(\beta)u(t^{n})+A_{2,2}(\beta)u(t^{n+\frac12})+A_{2,3}(\beta)u(t^{n+1}) \\
&=u_t(t^{n+\beta})\Delta t+C_{2,1}(n,\beta)\Delta t^3+C_{2,2}(n,\beta)\Delta t^4+C\Delta t^5,
\end{aligned}
\end{equation}
where the coefficients $A_{2,1}(\beta)$, $A_{2,2}(\beta)$, and $A_{2,3}(\beta)$ satisfy the following
\begin{equation}
\begin{bmatrix}
1 & 1 & 1 \\
\beta & \beta-\frac12 & \beta-1 \\
\beta^2 & (\beta-\frac12)^2 & (\beta-1)^2
\end{bmatrix}
\begin{bmatrix}
A_{2,1}(\beta) \\
A_{2,2}(\beta) \\
A_{2,3}(\beta)
\end{bmatrix}
=
\begin{bmatrix}
0 \\
-1 \\
0
\end{bmatrix}
.
\end{equation}
Considering that the above coefficient matrix is a Vandermonde matrix and $\beta\neq\beta-\frac12\neq\beta-1$ for any $\beta$, therefore its solution must be unique and exist. Upon completing the above matrix calculation, we can obtain
\begin{equation}
\label{A value}
\begin{aligned}
&A_{2,1}(\beta)=4\beta-3,\;A_{2,2}(\beta)=-8\beta+4,\;A_{2,3}(\beta)=4\beta-1,\\
&C_{2,1}(n,\beta)=\left(-\frac{\beta^2}2+\frac{\beta}2-\frac1{12}\right)u^{(3)}(t^{n+\beta}),\\
&C_{2,2}(n,\beta)=\left(\frac{\beta^3}3-\frac{\beta^2}2+\frac{11\beta}{48}-\frac{1}{32}\right)u^{(4)}(t^{n+\beta}).
\end{aligned}
\end{equation}

Similarly, we can also make the following equation hold true by modifying the coefficients of $u(t^{n})$, $u(t^{n+\frac12})$ and $u(t^{n+1})$:
\begin{equation}\label{eq3}
\begin{aligned}
&B_{2,1}(\beta)u(t^{n})+B_{2,2}(\beta)u(t^{n+\frac12})+B_{2,3}(\beta)u(t^{n+1}) \\
&=u(t^{n+\beta})+D_{2,1}(n,\beta)\Delta t^3+C\Delta t^4,
\end{aligned}
\end{equation}
where the coefficients $B_{2,1}(\beta)$, $B_{2,2}(\beta)$, and $B_{2,3}(\beta)$ satisfy 
\begin{equation}
\begin{bmatrix}
1 & 1 & 1 \\
\beta & \beta-\frac12 & \beta-1 \\
\beta^2 & (\beta-\frac12)^2 & (\beta-1)^2
\end{bmatrix}
\begin{bmatrix}
B_{2,1}(\beta) \\
B_{2,2}(\beta) \\
B_{2,3}(\beta)
\end{bmatrix}
=
\begin{bmatrix}
1 \\
0 \\
0
\end{bmatrix}
,
\end{equation}
and via matrix computation, one can derive
\begin{equation}
\label{A value}
\begin{aligned}
&B_{2,1}(\beta)=2\beta^2-3\beta+1,\;B_{2,2}(\beta)=-4\beta^2+4\beta,\;B_{2,3}(\beta)=2\beta^2-\beta,\\
&D_{2,1}(n,\beta)=\left(-\frac{\beta^3}6+\frac{\beta^2}4-\frac{\beta}{12}\right)u^{(3)}(t^{n+\beta}).
\end{aligned}
\end{equation}

Considering that $u(t^{n+\beta})$ is the exact solution of (\ref{eq1}) at time $t^{n+\beta}$, i.e.,
\begin{equation}\label{equal-system}
u_t(t^{n+\beta})+\mathcal Lu(t^{n+\beta})=f(t^{n+\beta}).
\end{equation}

By substituting equations (\ref{eq2}) and (\ref{eq3}) into above equation \eqref{equal-system}, we can obtain the following implicit difference formula
\begin{equation}\label{one-step-eq1}
\begin{aligned}
&\frac{1}{\Delta t}\left[A_{2,1}(\beta)u(t^{n})+A_{2,2}(\beta)u(t^{n+\frac12})+A_{2,3}(\beta)u(t^{n+1})\right] \\
&+\mathcal L\left[B_{2,1}(\beta)u(t^{n})+B_{2,2}(\beta)u(t^{n+\frac12})+B_{2,3}(\beta)u(t^{n+1})\right] \\
&=f(t^{n+\beta})+C_{2,1}(n,\beta)\Delta t^2+\left[C_{2,2}(n,\beta)+\mathcal LD_{2,1}(n,\beta)\right]\Delta t^3+C\Delta t^4
\end{aligned}
\end{equation}
Next, we'll introduce some symbols of operators to make our subsequent writing easier.
\begin{equation}
\begin{aligned}
&\mathcal{A}_\beta=\frac1{\Delta t}A_{2,2}(\beta)\mathcal I+B_{2,2}(\beta)\mathcal L,\;\mathcal{B}_\beta=\frac1{\Delta t}A_{2,3}(\beta)\mathcal I+B_{2,3}(\beta)\mathcal L, \\
&\mathcal{C}_\beta=\frac1{\Delta t}A_{2,1}(\beta)\mathcal I+B_{2,1}(\beta)\mathcal L.
\end{aligned}
\end{equation}
Then the equation \eqref{one-step-eq1} can be rewrite as the following formulation:
\begin{equation}\label{one-step-eq2}
\begin{aligned}
&\mathcal A_\beta u(t^{n+\frac12})+\mathcal B_\beta u(t^{n+1})+\mathcal C_\beta u(t^n)\\
&=f(t^{n+\beta})+C_{2,1}(n,\beta)\Delta t^2+\left[C_{2,2}(n,\beta)+\mathcal LD_{2,1}(n,\beta)\right]\Delta t^3+C\Delta t^4.
\end{aligned}
\end{equation}

Considering two distinct scalars, $\beta=\beta_1$ and $\beta=\beta_2$ ($\beta_1\neq\beta_2$) for above equation \eqref{one-step-eq2}, in the real number domain. We can obtain the following two equations:
\begin{equation}\label{one-step-eq3}
\begin{aligned}
&\mathcal{A}_{\beta_1}u(t^{n+\frac12})+\mathcal B_{\beta_1}u(t^{n+1})+\mathcal C_{\beta_1}u(t^n) \\
&=f(t^{n+\beta_1})+C_{2,1}(n,\beta_1)\Delta t^2+[C_{2,2}(n,\beta_1)+\mathcal LD_{2,1}(n,\beta_1)]\Delta t^3+C\Delta t^4.
\end{aligned}
\end{equation}
and
\begin{equation}\label{one-step-eq4}
\begin{aligned}
&\mathcal A_{\beta_2}u(t^{n+\frac12})+\mathcal B_{\beta_2}u(t^{n+1})+\mathcal C_{\beta_2}u(t^n) \\
&=f(t^{n+\beta_2})+C_{2,1}(n,\beta_2)\Delta t^2+[C_{2,2}(n,\beta_2)+\mathcal LD_{2,1}(n,\beta_2)]\Delta t^3+C\Delta t^4.
\end{aligned}
\end{equation}
Multiplying the equation \eqref{one-step-eq3} by $\mathcal A_{\beta_2}$, equation \eqref{one-step-eq4} by $-\mathcal A_{\beta_1}$ and then add them together to eliminate the unknown term $u(t^{n+\frac12})$, then we can obtain the following equation:
\begin{equation}\label{eq4}
\begin{aligned}
&\left(\mathcal A_{\beta_2}\mathcal B_{\beta_1}-\mathcal A_{\beta_1}\mathcal B_{\beta_2}\right)u(t^{n+1}) \\
=&\left(\mathcal A_{\beta_1}\mathcal C_{\beta_2}-\mathcal A_{\beta_2}\mathcal C_{\beta_1}\right)u(t^n)+\mathcal A_{\beta_2}f(t^{n+\beta_1})-\mathcal A_{\beta_1}f(t^{n+\beta_2})\\
&+[\mathcal A_{\beta_2}C_{2,1}(n,\beta_1)-\mathcal A_{\beta_1}C_{2,1}(n,\beta_2)]\Delta t^2 \\
&+\left\{\mathcal A_{\beta_2}[C_{2,2}(n,\beta_1)+\mathcal LD_{2,1}(n,\beta_1)]-\mathcal A_{\beta_1}[C_{2,2}(n,\beta_2)+\mathcal LD_{2,1}(n,\beta_2)]\right\}\Delta t^3\\
&+C\Delta t^4.
\end{aligned}
\end{equation}
Here $\mathcal A_{\beta_2}\mathcal B_{\beta_1}-\mathcal A_{\beta_1}\mathcal B_{\beta_2}$ is obviously an invertible operator because of
\begin{equation}
\begin{aligned}
\mathcal A_{\beta_2}\mathcal B_{\beta_1}-\mathcal A_{\beta_1}\mathcal B_{\beta_2}=\frac{4(\beta_1-\beta_2)[\beta_1\beta_2\Delta t^2\mathcal L^2+(\beta_1+\beta_2)\Delta t\mathcal L+2\mathcal I]}{\Delta t^2}.
\end{aligned}
\end{equation}
Then, a new class of high-order one-step schemes for system \eqref{eq1} are
\begin{equation}\label{eq5}
\begin{aligned}
[\mathcal A_{\beta_2}\mathcal B_{\beta_1}-\mathcal A_{\beta_1}\mathcal B_{\beta_2}]u^{n+1}=
&[\mathcal A_{\beta_1}\mathcal C_{\beta_2}-\mathcal A_{\beta_2}\mathcal C_{\beta_1}]u^n+\mathcal A_{\beta_2}f(t^{n+\beta_1})-\mathcal A_{\beta_1}f(t^{n+\beta_2}).
\end{aligned}
\end{equation}

Next, we will discuss the truncation errors of above one-step schemes. Define the following truncation errors as
\begin{equation}
E^{n+1}=u(t^{n+1})-\mathcal M_{\beta_1,\beta_2}u(t^{n+1}),
\end{equation}
where
\begin{equation}
\begin{aligned}
\mathcal M_{\beta_1,\beta_2}u(t^{n+1})=
&2\mathcal F_{\beta_1,\beta_2}u(t^{n})+(\beta_1+\beta_2-2)\Delta t\mathcal L\mathcal F_{\beta_1,\beta_2}u(t^{n})\\
&+(\beta_1\beta_2-\beta_1-\beta_2+1)\Delta t^2\mathcal L^2\mathcal F_{\beta_1,\beta_2}u(t^{n})\\
&+\mathcal G_{\beta_1,\beta_2}\Delta tf^{n+\beta_1}+\mathcal G_{\beta_2,\beta_1}\Delta tf^{n+\beta_2}.
\end{aligned}
\end{equation}
Here the operators $\mathcal F_{\beta_1,\beta_2}$ and $\mathcal G_{\beta_1,\beta_2}$ are defined by the following
\begin{equation}
\begin{aligned}
\mathcal F_{\beta_1,\beta_2}=&[\beta_1\beta_2\Delta t^2\mathcal L^2+(\beta_1+\beta_2)\Delta t\mathcal L+2\mathcal I]^{-1}, \\
\mathcal G_{\beta_1,\beta_2}=&\frac{(2\beta_2-1)\mathcal I+\beta_2(\beta_2-1)\Delta t\mathcal L}{\beta_1-\beta_2}\mathcal F_{\beta_1,\beta_2}.
\end{aligned}
\end{equation}

Combining the equation \eqref{eq4} with \eqref{eq5}, we are easy to obtain a detailed expression of $E^{n+1}$: 
\begin{equation}\label{expandeq}
\begin{aligned}
u(t^{n+1})
&=\mathcal M_{\beta_1,\beta_2}u(t^{n+1})+E^{n+1}\\
&=\mathcal M_{\beta_1,\beta_2}u(t^{n+1})+\frac{6\beta_1^2-6\beta_1+1}{12}\mathcal G_{\beta_1,\beta_2}u^{(3)}(t^{n+\beta_1})\Delta t^3 \\
&\quad+\frac{6\beta_2^2-6\beta_2+1}{12}\mathcal G_{\beta_2,\beta_1}u^{(3)}(t^{n+\beta_2})\Delta t^3 \\
&\quad-\frac{32\beta_1^3-48\beta_1^2+22\beta_1-3}{96}\mathcal G_{\beta_1,\beta_2}u^{(4)}(t^{n+\beta_1})\Delta t^4 \\
&\quad-\frac{32\beta_2^3-48\beta_2^2+22\beta_2-3}{96}\mathcal G_{\beta_2,\beta_1}u^{(4)}(t^{n+\beta_2})\Delta t^4 \\
&\quad+\frac{2\beta_1^3-3\beta_1^2+\beta_1}{12}\mathcal G_{\beta_1,\beta_2}\mathcal Lu^{(3)}(t^{n+\beta_1})\Delta t^4 \\
&\quad+\frac{2\beta_2^3-3\beta_2^2+\beta_2}{12}\mathcal G_{\beta_2,\beta_1}\mathcal Lu^{(3)}(t^{n+\beta_2})\Delta t^4+C\Delta t^5.
\end{aligned}
\end{equation}

To facilitate the representation and subsequent analysis, we introduce the following six operators:   
\begin{equation}
\begin{aligned}
&\mathcal H_{1,\beta_1,\beta_2}=&\frac{6\beta_1^2-6\beta_1+1}{12}\mathcal G_{\beta_1,\beta_2},\quad\mathcal H_{2,\beta_1,\beta_2}&=\frac{6\beta_2^2-6\beta_2+1}{12}\mathcal G_{\beta_2,\beta_1}, \\
&\mathcal H_{3,\beta_1,\beta_2}=&-\frac{32\beta_1^3-48\beta_1^2+22\beta_1-3}{96}\mathcal G_{\beta_1,\beta_2},& \\
&\mathcal H_{4,\beta_1,\beta_2}=&-\frac{32\beta_2^3-48\beta_2^2+22\beta_2-3}{96}\mathcal G_{\beta_2,\beta_1},& \\
&\mathcal H_{5,\beta_1,\beta_2}=&\frac{2\beta_1^3-3\beta_1^2+\beta_1}{12}\mathcal G_{\beta_1,\beta_2},\quad\mathcal H_{6,\beta_1,\beta_2}&=\frac{2\beta_2^3-3\beta_2^2+\beta_2}{12}\mathcal G_{\beta_2,\beta_1}.
\end{aligned}
\end{equation}

Obviously, if we do not add any restrictions on the parameter $\beta_1$ and $\beta_2$, $E^{n+1}$ will satisfy the following
\begin{equation}
E^{n+1}=\mathcal O(\Delta t^3).
\end{equation}
Nevertheless, in actuality, when $\beta_1$ and $\beta_2$ meet specific criteria, we are able to achieve a more favorable truncation error.
\begin{theorem}\label{thm2}
If $\beta_1$ and $\beta_2$ satisfy $3(2\beta_1-1)(2\beta_2-1)=-1$ and $|\beta_2-\beta_1|\ll N$, we can obtain
\begin{equation}
E^{n+1}=\mathcal O(\Delta t^4).
\end{equation}
Specially, if we set $\beta_1=\frac{3+\sqrt3}{6},$ and $\beta_2=\frac{3-\sqrt3}{6}$ (or $\beta_1=\frac{3-\sqrt3}{6},$ and $\beta_2=\frac{3+\sqrt3}{6}$), we can obtain
\begin{equation}
E^{n+1}=\mathcal O(\Delta t^5).
\end{equation}

\end{theorem}

\begin{proof}
Based on the Taylor expansion, if $|\beta_2-\beta_1|\ll N$, we have
\begin{equation}
u^{(3)}(t^{n+\beta_2})=u^{(3)}(t^{n+\beta_1})+C\Delta t,\;u^{(4)}(t^{n+\beta_2})=u^{(4)}(t^{n+\beta_1})+C\Delta t.
\end{equation}
If $\beta_1$ and $\beta_2$ satisfy $3(2\beta_1-1)(2\beta_2-1)=-1$, the sum of the first two terms of the truncation error $E^{n+1}$ satisfies
\begin{equation}
\begin{aligned}
&\mathcal H_{1,\beta_1,\beta_2}u^{(3)}(t^{n+\beta_1})\Delta t^3+\mathcal H_{2,\beta_1,\beta_2}u^{(3)}(t^{n+\beta_2})\Delta t^3 \\
=&\left(\mathcal H_{1,\beta_1,\beta_2}+\mathcal H_{2,\beta_1,\beta_2}\right)u^{(3)}(t^{n+\beta_1})\Delta t^3+C\Delta t^4 \\
=&\frac{\mathcal F_{\beta_1,\beta_2}}{12(\beta_1-\beta_2)}\left((6\beta_1^2-6\beta_1+1)(2\beta_2-1)-(6\beta_2^2-6\beta_2+1)(2\beta_1-1)\right)u^{(3)}(t^{n+\beta_1})\Delta t^3 \\
&+\frac{\mathcal F_{\beta_1,\beta_2}}{12(\beta_1-\beta_2)}\left((6\beta_1^2-6\beta_1+1)\beta_2(\beta_2-1)-(6\beta_2^2-6\beta_2+1)\beta_1(\beta_1-1)\right) \\
&\times\mathcal Lu^{(3)}(t^{n+\beta_1})\Delta t^4+C\Delta t^4 \\
=&\frac{\mathcal F_{\beta_1,\beta_2}}{12}\left(\left(3(2\beta_1-1)(2\beta_2-1)+1\right)\mathcal I+(1-\beta_1-\beta_2)\Delta t\mathcal L\right)u^{(3)}(t^{n+\beta_1})\Delta t^3+C\Delta t^4 \\
=&\frac{\mathcal F_{\beta_1,\beta_2}}{12}(1-\beta_1-\beta_2)\mathcal Lu^{(3)}(t^{n+\beta_1})\Delta t^4+C\Delta t^4,
\end{aligned}
\end{equation}
which means
\begin{equation}
E^{n+1}=\mathcal O(\Delta t^4).
\end{equation}
For example, we can choose $\beta_1=1$ and $\beta_2=\frac13$ or $\beta_1=0$ and $\beta_2=\frac23$.

Without loss of generality, we set $\beta_1=\frac{3+\sqrt3}{6},$ and $\beta_2=\frac{3-\sqrt3}{6}$ to satisfy $6\beta_1^2-6\beta_1+1=0$ and $6\beta_2^2-6\beta_2+1=0$. It leads to 
\begin{equation}\label{one-step-eq5}
\mathcal H_{1,\beta_1,\beta_2}=0,\;\mathcal H_{2,\beta_1,\beta_2}=0.
\end{equation}
Then the sum of the third and fourth terms of the truncation error $E^{n+1}$ satisfies
\begin{equation}\label{one-step-eq6}
\begin{aligned}
&\mathcal H_{3,\beta_1,\beta_2}u^{(4)}(t^{n+\beta_1})\Delta t^4+\mathcal H_{4,\beta_1,\beta_2}u^{(4)}(t^{n+\beta_2})\Delta t^4\\
&=(\mathcal H_{3,\beta_1,\beta_2}+\mathcal H_{4,\beta_1,\beta_2})u^{(4)}(t^{n+\beta_1})\Delta t^4+C\Delta t^5 \\
&=\frac1{864}(\frac16\Delta t^2\mathcal L^2+\Delta t\mathcal L+2\mathcal I)^{-1}\mathcal Lu^{(4)}(t^{n+\beta_1})\Delta t^5+C\Delta t^5,
\end{aligned}
\end{equation}
and the sum of the fifth and sixth terms of the truncation error $E^{n+1}$ satisfies
\begin{equation}\label{one-step-eq7}
\begin{aligned}
&\mathcal H_{5,\beta_1,\beta_2}\mathcal Lu^{(3)}(t^{n+\beta_1})\Delta t^4+\mathcal H_{6,\beta_1,\beta_2}\mathcal Lu^{(3)}(t^{n+\beta_2})\Delta t^4 \\
&=(\mathcal H_{5,\beta_1,\beta_2}+\mathcal H_{6,\beta_1,\beta_2})\mathcal Lu^{(3)}(t^{n+\beta_1})\Delta t^4+C\Delta t^5 \\
&=\frac1{162}(\frac16\Delta t^2\mathcal L^2+\Delta t\mathcal L+2\mathcal I)^{-1}\mathcal L^2u^{(3)}(t^{n+\beta_1})\Delta t^5+C\Delta t^5,
\end{aligned}
\end{equation}
Combining the equation \cref{one-step-eq5} with \cref{one-step-eq6}-\cref{one-step-eq7}, we are able to derive
\begin{equation}
E^{n+1}=\mathcal O(\Delta t^5).
\end{equation}
\end{proof}

Next, we will show that the proposed one-step scheme \eqref{eq5} degenerates to classical Runge-Kutta scheme when $\beta_1$ and $\beta_2$ assume specific values.
\begin{theorem}\label{thm1}
Under the condition of $\beta_1=0$ and arbitrary $\beta_2\neq\beta_1$, the proposed one-step scheme \eqref{eq5} will degenerate to a Runge-Kutta scheme with stability terms. Specially, if the parabolic system \eqref{eq1} assumes the following canonical form:
\begin{equation}\label{eq-system}
\begin{aligned}
&u_t=f(t,u),\;0<t<T,\\
&u(0)=u_0
\end{aligned}
\end{equation}
the proposed one-step scheme \eqref{eq5} at $\beta_1=0$ degenerates exactly to the classical RK2 scheme.
\end{theorem}

\begin{proof}
When $\beta_1=0$ and for any $\beta_2\neq\beta_1$, the scheme \eqref{eq5} reduces to
\begin{equation}
\begin{aligned}
(\beta_2\Delta t\mathcal L+2\mathcal I)u^{n+1}=&\left[(1-\beta_2)\Delta t^2\mathcal L^2+(\beta_2-2)\Delta t\mathcal L+2\mathcal I\right]u^{n} \\
&+\frac{(1-2\beta_2)+\beta_2(1-\beta_2)\Delta t\mathcal L}{2\beta_2}\Delta tf^{n}+\frac{1}{2\beta_2}\Delta tf^{n+\beta_2}.
\end{aligned}
\end{equation}

Through elementary algebraic manipulation, the preceding equation is reformulated as
\begin{equation}\label{mRK_2}
\begin{aligned}
&u^{n+1}+\frac{1}{2}\beta_2\Delta t\mathcal L\left[u^{n+1}-(u^n-\Delta t\mathcal Lu^n+f^n)\right]\\
=&u^n-\Delta t\mathcal Lu^n+\frac12\Delta t^2\mathcal L^2u^n+\frac{2\beta_2-1}{2\beta_2}\Delta tf^{n}-\frac12\Delta t^2\mathcal Lf^n+\frac{1}{2\beta_2}\Delta tf^{n+\beta_2}.
\end{aligned}
\end{equation}

Based on \eqref{RK_scheme}, we derive the following RK1 and RK2 schemes:

RK1:
\begin{equation}
u^{n+1}=u^n-\Delta t\mathcal Lu^n+f^n
\end{equation}

RK2:
\begin{equation}
\begin{aligned}
&u^{n+1}=u^n-\Delta t\mathcal Lu^n+\frac12\Delta t^2\mathcal L^2u^n+d_1\Delta tf^{n}-\frac12\Delta t^2\mathcal Lf^n+d_2\Delta tf^{n+e_2}, \\
&d_1+d_2=1,\;d_2e_2=\frac12.
\end{aligned}
\end{equation}

An easy comparative analysis between the RK2 scheme and the proposed one-step formulation \eqref{mRK_2}, along with the identities below:
\begin{equation}
\frac{2\beta_2-1}{2\beta_2}+\frac{1}{2\beta_2}=1,\;\frac{1}{2\beta_2}\cdot\beta_2=\frac12,
\end{equation}
reveals that when $\beta_1=0$, the one-step scheme \eqref{mRK_2} incorporates a second-order stabilization term while maintaining equivalence to the classical RK2 scheme. It is rigorously established by subsequent stability analysis that the augmentation with stabilization terms substantially improves the stability region of the RK2 scheme.

We now prove that if the parabolic equation \eqref{eq1} takes form \eqref{eq-system}, then the proposed one-step scheme \eqref{mRK_2} with $\beta = 0$ reduces exactly to the RK2 method. Based on the previous analysis, the one-step scheme \eqref{mRK_2} for system \eqref{eq-system} will become 
\begin{equation}
u^{n+1}=u^n+\Delta t\frac{(1-2\beta_2)f(t^{n+\beta_1},\hat u^{n+\beta_1})-(1-2\beta_1)f(t^{n+\beta_2},\hat u^{n+\beta_2})}{2(\beta_1-\beta_2)},
\end{equation}
where $\hat u^{n+\beta_1}$ and $\hat u^{n+\beta_2}$ serve as an approximation of $u(t^{n+\beta_1})$ and $u(t^{n+\beta_2})$.

When we set $\beta_1=0$ and $\hat u^{n+\beta_2}=u^n+\beta_2\Delta tf(t^n,u^n)$, the above equation degenerates into
\begin{equation}
u^{n+1}=u^n+\frac{2\beta_2-1}{2\beta_2}\Delta tf(t^n,u^{n})+\frac{1}{2\beta_2}\Delta tf\left(t^{n+\beta_2},u^n+\beta_2\Delta tf(t^n,u^n)\right),
\end{equation}
which is totally RK2 scheme for system \eqref{eq-system}.
\end{proof}
 
\subsection{The second high-order one-step scheme}
In developing the high=order one-step scheme in the preceding subsection, we introduced only one intermediate layer $t^{n+\frac12}$. By employing two distinct $\beta$ parameters on this basis to eliminate the unknown term $u(t^{n+\frac12})$, we derived the second-order one-step scheme \eqref{eq5}. It is clear that introducing $k$ intermediate layers $t^{n+\frac{1}{k+1}}$, $t^{n+\frac{2}{k+1}}$, $\ldots$, $t^{n+\frac{k}{k+1}}$, coupled with $k+1$ distinct $\beta$ parameters to eliminate $k$ unknowns $u(t^{n+\frac{1}{k+1}})$, $u(t^{n+\frac{2}{k+1}})$, $\ldots$, $u(t^{n+\frac{k}{k+1}})$, allows us to establish a general procedure for constructing $(k+1)$th-order one-step schemes. In the subsection, we will establish the one-step third-order scheme for the representative case $k=2$, achieved by implementing two intermediate layers and three distinct $\beta$ parameters within the framework.

Performing Taylor expansions of $u(t^{n})$, $u(t^{n+1})$ and the two intermediate layers $u(t^{n+\frac13})$, $u(t^{n+\frac23})$ with respect to the reference node $t^{n+\beta}$ yields the following system of equations:
\begin{equation}
\label{taylor3}
\begin{aligned}
u(t^{n+1})=&u(t^{n+\beta})-u_t(t^{n+\beta})(\beta-1)\Delta t+u^{(2)}(t^{n+\beta})\frac{(\beta-1)^2\Delta t^2}{2} \\
&-u^{(3)}(t^{n+\beta})\frac{(\beta-1)^3\Delta t^3}{6}+u^{(4)}(t^{n+\beta})\frac{(\beta-1)^4\Delta t^4}{24}+C\Delta t^5, \\
u(t^{n+\frac23})=&u(t^{n+\beta})-u_t(t^{n+\beta})(\beta-\frac23)\Delta t+u^{(2)}(t^{n+\beta})\frac{(\beta-\frac23)^2\Delta t^2}{2} \\
&-u^{(3)}(t^{n+\beta})\frac{(\beta-\frac23)^3\Delta t^3}{6}+u^{(4)}(t^{n+\beta})\frac{(\beta-\frac23)^4\Delta t^4}{24}+C\Delta t^5, \\
u(t^{n+\frac13})=&u(t^{n+\beta})-u_t(t^{n+\beta})(\beta-\frac13)\Delta t+u^{(2)}(t^{n+\beta})\frac{(\beta-\frac13)^2\Delta t^2}{2} \\
&-u^{(3)}(t^{n+\beta})\frac{(\beta-\frac13)^3\Delta t^3}{6}+u^{(4)}(t^{n+\beta})\frac{(\beta-\frac13)^4\Delta t^4}{24}+C\Delta t^5, \\
u(t^{n})=&u(t^{n+\beta})-u_t(t^{n+\beta})\beta\Delta t+u^{(2)}(t^{n+\beta})\frac{\beta^2\Delta t^2}{2} \\
&-u^{(3)}(t^{n+\beta})\frac{\beta^3\Delta t^3}{6}+u^{(4)}(t^{n+\beta})\frac{\beta^4\Delta t^4}{24}+C\Delta t^5.
\end{aligned}
\end{equation}

Through algebraic manipulation of the aforementioned four equations \eqref{taylor3}, the first, third and fourth terms of right hand for every equations  can be annihilated, yielding the following expression:
\begin{equation}
\label{eqthird}
\begin{aligned}
&A_{3,1}(\beta)u(t^{n})+A_{3,2}(\beta)u(t^{n+\frac13})+A_{3,3}(\beta)u(t^{n+\frac23})+A_{3,4}(\beta)u(t^{n+1}) \\
&=u_t(t^{n+\beta})\Delta t+C_{3,1}(n,\beta)\Delta t^4+C\Delta t^5,
\end{aligned}
\end{equation}
where its coefficients $A_{3,1}(\beta),\;A_{3,2}(\beta),\;A_{3,3}(\beta)$ and $A_{3,4}(\beta)$ satisfy 
\begin{equation}\label{eq8}
\begin{bmatrix}
1 & 1 & 1 & 1\\
\beta & \beta-\frac13 & \beta-\frac23 & \beta-1 \\
\beta^2 & (\beta-\frac13)^2 & (\beta-\frac23)^2 & (\beta-1)^2 \\
\beta^3 & (\beta-\frac13)^3 & (\beta-\frac23)^3 & (\beta-1)^3
\end{bmatrix}
\begin{bmatrix}
A_{3,1}(\beta) \\
A_{3,2}(\beta) \\
A_{3,3}(\beta) \\
A_{3,4}(\beta)
\end{bmatrix}
=
\begin{bmatrix}
0 \\
-1 \\
0 \\
0
\end{bmatrix}
.
\end{equation}
The solution to the matrix system \eqref{eq8} is directly accessible through the non-singularity of the Vandermonde structure and basic linear algebraic manipulations, yielding the explicit expression:
\begin{equation}
\begin{aligned}
&A_{3,1}(\beta)=-\frac{27\beta^2}2+18\beta-\frac{11}2,\;A_{3,2}(\beta)=\frac{81\beta^2}2-45\beta+9, \\
&A_{3,3}(\beta)=-\frac{81\beta^2}2+36\beta-\frac{9}2,\;A_{3,4}(\beta)=\frac{27\beta^2}2-9\beta+1, \\
&C_{3,1}(n,\beta)=\left(-\frac{\beta^3}6+\frac{\beta^2}4-\frac{11\beta}{108}+\frac{1}{108}\right)u^{(4)}(t^{n+\beta}).
\end{aligned}
\end{equation}

Similarly, resolving the linear system formed by combining the four equations \eqref{taylor3} eliminates the coefficients of the second, third and fourth terms of right hand for every equations, resulting in
\begin{equation}\label{eqthird2}
\begin{aligned}
&B_{3,1}(\beta)u(t^{n})+B_{3,2}(\beta)u(t^{n+\frac13})+B_{3,3}(\beta)u(t^{n+\frac23})+B_{3,4}(\beta)u(t^{n+1})\\
=&u(t^{n+\beta})+C\Delta t^4,
\end{aligned}
\end{equation}
where $B_{3,1}(\beta)$, $B_{3,2}(\beta)$, $B_{3,3}(\beta)$, and $B_{3,4}(\beta)$ satisfy 
\begin{equation}\label{eq9}
\begin{bmatrix}
1 & 1 & 1 & 1\\
\beta & \beta-\frac13 & \beta-\frac23 & \beta-1 \\
\beta^2 & (\beta-\frac13)^2 & (\beta-\frac23)^2 & (\beta-1)^2 \\
\beta^3 & (\beta-\frac13)^3 & (\beta-\frac23)^3 & (\beta-1)^3
\end{bmatrix}
\begin{bmatrix}
B_{3,1}(\beta) \\
B_{3,2}(\beta) \\
B_{3,3}(\beta) \\
B_{3,4}(\beta)
\end{bmatrix}
=
\begin{bmatrix}
1 \\
0	 \\
0 \\
0
\end{bmatrix}
.
\end{equation}

Leveraging the invertibility of the Vandermonde matrix coupled with elementary matrix operations, we readily solve the aforementioned matrix system \eqref{eq9} to obtain the solution as follows:
\begin{equation}
\begin{aligned}
&B_{3,1}(\beta)=-\frac{9\beta^3}2+9\beta^2-\frac{11}2\beta+1,\;B_{3,2}(\beta)=\frac{27\beta^3}2-\frac{45\beta^2}2+9\beta, \\
&B_{3,3}(\beta)=-\frac{27\beta^3}2+18\beta^2-\frac{9\beta}2,\;B_{3,4}(\beta)=\frac{9\beta^3}2-\frac{9\beta^2}2+\beta. \\
\end{aligned}
\end{equation}

By substituting equations (\ref{eqthird}) and (\ref{eqthird2}) into equation (\ref{eq1}), we can obtain
\begin{equation}\label{eq10}
\begin{aligned}
&\frac{1}{\Delta t}\left[A_{3,1}(\beta)u(t^{n})+A_{3,2}(\beta)u(t^{n+\frac13})+A_{3,3}(\beta)u(t^{n+\frac23})+A_{3,4}(\beta)u(t^{n+1})\right] \\
&+\mathcal L\left[B_{3,1}(\beta)u(t^{n})+B_{3,2}(\beta)u(t^{n+\frac13})+B_{3,3}(\beta)u(t^{n+\frac23})+B_{3,4}(\beta)u(t^{n+1})\right] \\
&=f(t^{n+\beta})+C_{3,1}(n,\beta)\Delta t^3+C\Delta t^4.
\end{aligned}
\end{equation}

Subsequently, to streamline the exposition, we introduce the following four operator notations:
\begin{equation}\label{eq11}
\begin{aligned}
&\mathcal A_{\beta}=\frac1{\Delta t}A_{3,2}(\beta)\mathcal I+B_{3,2}(\beta)\mathcal L,\;\mathcal B_{\beta}=\frac1{\Delta t}A_{3,3}(\beta)\mathcal I+B_{3,3}(\beta)\mathcal L, \\
&\mathcal C_{\beta}=\frac1{\Delta t}A_{3,4}(\beta)\mathcal I+B_{3,4}(\beta)\mathcal L,\;\mathcal D_{\beta}=-\frac1{\Delta t}A_{3,1}(\beta)\mathcal I-B_{3,1}(\beta)\mathcal L.
\end{aligned}
\end{equation}

Consider three distinct real scalars $\beta_1$, $\beta_2$ and $\beta_3$ ($\beta_1\neq\beta_2\neq\beta_3$) for equation \eqref{eq10} to obtain
\begin{equation}
\begin{aligned}
\begin{bmatrix}
\mathcal A_{\beta_1} & \mathcal B_{\beta_1} & \mathcal C_{\beta_1} \\
\mathcal A_{\beta_2} & \mathcal B_{\beta_2} & \mathcal C_{\beta_2} \\
\mathcal A_{\beta_3} & \mathcal B_{\beta_3} & \mathcal C_{\beta_3} 
\end{bmatrix}
\begin{bmatrix}
u(t^{n+\frac13}) \\
u(t^{n+\frac23}) \\
u(t^{n+1})
\end{bmatrix}
=
\begin{bmatrix}
\mathcal D_{\beta_1}u(t^{n})+f(t^{n+\beta_1})+C_{3,1}(n,\beta_1)\Delta t^3 \\
\mathcal D_{\beta_2}u(t^{n})+f(t^{n+\beta_2})+C_{3,1}(n,\beta_2)\Delta t^3 \\
\mathcal D_{\beta_3}u(t^{n})+f(t^{n+\beta_3})+C_{3,1}(n,\beta_3)\Delta t^3
\end{bmatrix}
+C\Delta t^4.
\end{aligned}
\end{equation}

Based on the explicit operator definitions in \eqref{eq11}, the analytic expression of the determinant operator of the coefficient matrix components:
\begin{equation}
\begin{aligned}
\left\vert
\begin{matrix}
\mathcal A_{\beta_1} & \mathcal B_{\beta_1} & \mathcal C_{\beta_1} \\
\mathcal A_{\beta_2} & \mathcal B_{\beta_2} & \mathcal C_{\beta_2} \\
\mathcal A_{\beta_3} & \mathcal B_{\beta_3} & \mathcal C_{\beta_3} 
\end{matrix}
\right\vert=&-\frac{243(\beta_1-\beta_2)(\beta_1-\beta_3)(\beta_2-\beta_3)}{4\Delta t^3}\left[\beta_1\beta_2\beta_3\Delta t^3\mathcal L^3\right. \\
&+\left.(\beta_1\beta_2+\beta_1\beta_3+\beta_2\beta_3)\Delta t^2\mathcal L^2+2(\beta_1+\beta_2+\beta_3)\Delta t\mathcal L+6\mathcal I\right].
\end{aligned}
\end{equation}
Given that $\beta_1$, $\beta_2$ and $\beta_3$ are pairwise distinct, it is evident that the above operator is invertible. Consequently, via matrix computation, we obtain the representation of $u(t^{n+1})$ as follows:
\begin{equation}\label{eq12}
\begin{aligned}
u(t^{n+1})=&[(\beta_1-1)\Delta t\mathcal L+\mathcal I][(\beta_2-1)\Delta t\mathcal L+\mathcal I][(\beta_3-1)\Delta t\mathcal L+\mathcal I]\mathcal F_{\beta_1,\beta_2,\beta_3}u(t^n) \\
&+[(\beta_1+\beta_2+\beta_3-3)\Delta t\mathcal L+5\mathcal I]\mathcal F_{\beta_1,\beta_2,\beta_3}u(t^n) \\
&+\mathcal G_{\beta_1,\beta_2,\beta_3}\Delta tf(t^{n+\beta_1})+\mathcal G_{\beta_2,\beta_3,\beta_1}\Delta tf(t^{n+\beta_2})+\mathcal G_{\beta_3,\beta_1,\beta_2}\Delta tf(t^{n+\beta_3}) \\
&-\frac{18\beta_1^3-27\beta_1^2+11\beta_1-1}{108}\mathcal G_{\beta_1,\beta_2,\beta_3}u^{(4)}(t^{n+\beta_1})\Delta t^4 \\
&-\frac{18\beta_2^3-27\beta_2^2+11\beta_2-1}{108}\mathcal G_{\beta_2,\beta_3,\beta_1}u^{(4)}(t^{n+\beta_2})\Delta t^4 \\
&-\frac{18\beta_3^3-27\beta_3^2+11\beta_3-1}{108}\mathcal G_{\beta_3,\beta_1,\beta_2}u^{(4)}(t^{n+\beta_1})\Delta t^4 +C\Delta t^5,
\end{aligned}
\end{equation}
where 
\begin{equation}
\begin{aligned}
\mathcal F_{\beta_1,\beta_2,\beta_3}=&[(\beta_1\Delta t\mathcal L+\mathcal I)(\beta_2\Delta t\mathcal L+\mathcal I)(\beta_3\Delta t\mathcal L+\mathcal I)+(\beta_1+\beta_2+\beta_3)\Delta t\mathcal L+5\mathcal I]^{-1}, \\
\mathcal G_{\beta_1,\beta_2,\beta_3}=&\frac{\beta_2\beta_3(\beta_2-1)(\beta_3-1)\Delta t^2\mathcal L^2}{(\beta_1-\beta_2)(\beta_1-\beta_3)}\mathcal F_{\beta_1,\beta_2,\beta_3} \\
&+\frac{(2\beta_2\beta_3-\beta_2-\beta_3)(\beta_2+\beta_3-1)\Delta t\mathcal L}{(\beta_1-\beta_2)(\beta_1-\beta_3)}\mathcal F_{\beta_1,\beta_2,\beta_3} \\
&+\frac{6\beta_2\beta_3-3\beta_2-3\beta_3+2}{(\beta_1-\beta_2)(\beta_1-\beta_3)}\mathcal F_{\beta_1,\beta_2,\beta_3}.
\end{aligned}
\end{equation}

Consequently, we derive the one-step third-order scheme as follows:
\begin{equation}\label{eq13}
\begin{aligned}
u^{n+1}=&[(\beta_1-1)\Delta t\mathcal L+\mathcal I][(\beta_2-1)\Delta t\mathcal L+\mathcal I][(\beta_3-1)\Delta t\mathcal L+\mathcal I]\mathcal F_{\beta_1,\beta_2,\beta_3}u^n \\
&+[(\beta_1+\beta_2+\beta_3-3)\Delta t\mathcal L+5\mathcal I]\mathcal F_{\beta_1,\beta_2,\beta_3}u^n \\
&+\mathcal G_{\beta_1,\beta_2,\beta_3}\Delta tf(t^{n+\beta_1})+\mathcal G_{\beta_2,\beta_3,\beta_1}\Delta tf(t^{n+\beta_2})+\mathcal G_{\beta_3,\beta_1,\beta_2}\Delta tf(t^{n+\beta_3}).
\end{aligned}
\end{equation}

By fixing the values of $\beta_1$, $\beta_2$ and $\beta_3$, we can achieve higher orders of convergence; for example: 
\begin{remark}\label{rem2}
Under the parameter specification $(\beta_1,\beta_2,\beta_3)=(\frac12,\frac{3+\sqrt{5}}{6},\frac{3-\sqrt{5}}{6})$, the relation $18\beta_i^3-27\beta_i^2+11\beta_i-1=0$ $(i=1,2,3)$ holds identically. This forces the fourth, fifth, and sixth terms on the right-hand side of \eqref{eq12} to be zero, resulting in a dominant truncation error of order $O(\Delta t^5)$. Specially, if any one of the parameters $\beta_1$, $\beta_2$, or $\beta_3$ is zero, this scheme \eqref{eq13} degenerates into a third-order Runge-Kutta method with stability terms.
\end{remark}
\section{Stability analysis}
In this section, we conduct stability analysis of the proposed one-step schemes, demonstrating that A-stability and L-stability are achieved when $\beta_i$ satisfies specific conditions.
\subsection{Linear stability regions of \eqref{eq5}}
For the test equation $u_t=\lambda u$, (\ref{eq5}) reduces to
\begin{equation}\label{eq7}
\begin{aligned}
&\left[\beta_1\beta_2\lambda^2\Delta t^2-(\beta_1+\beta_2)\lambda\Delta t+2\right]u^{n+1}\\
=&\left[(\beta_1\beta_2-\beta_1-\beta_2+1)\lambda^2\Delta t^2+(2-\beta_1-\beta_2)\lambda\Delta t+2\right]u^n.
\end{aligned}
\end{equation}
In order to study the stability regions, we set $z=\lambda\Delta t$ and 
\begin{equation}
\begin{aligned}
g_1(z)&=(\beta_1\beta_2-\beta_1-\beta_2+1)z^2+(2-\beta_1-\beta_2)z+2, \\
g_2(z)&=\beta_1\beta_2z^2-(\beta_1+\beta_2)z+2.
\end{aligned}
\end{equation}
Then the region of absolute stability of method (\ref{eq7}) is the set of all $z\in\mathcal C$ such that  $g_2(z)\neq0$ and $\left\vert\frac{g_1(z)}{g_2(z)}\right\vert\leq1$. In Fig.~\ref{stable_ex1}, \ref{stable_ex3}, and \ref{stable_ex4}, we plot the stability regions of the new one-step scheme \eqref{eq5} with different $\beta_1$ and $\beta_2$. It can be observed that for any fixed $\beta_1$, an A-stable one-step scheme can be systematically achieved through deliberate adjustment of $\beta_2$.

\begin{figure}[htp]
\centering
\begin{minipage}{0.32\linewidth}
\centerline{\includegraphics[width=\textwidth]{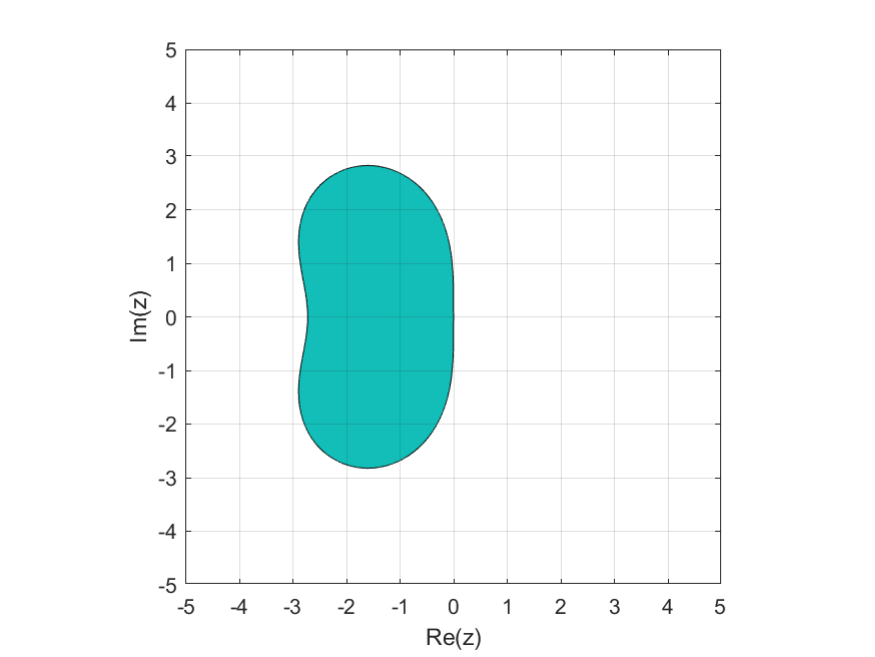}}
\centerline{$\beta_1=-\frac13,\beta_2=\frac35$}
\end{minipage}
\begin{minipage}{0.32\linewidth}
\centerline{\includegraphics[width=\textwidth]{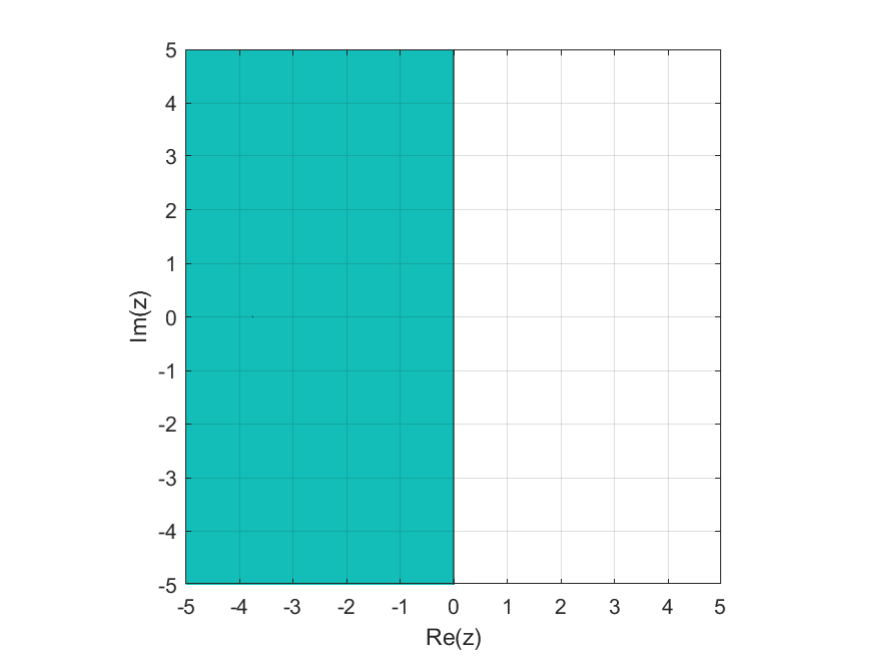}}
\centerline{$\beta_1=-\frac13,\beta_2=\frac45$}
\end{minipage}
\begin{minipage}{0.32\linewidth}
\centerline{\includegraphics[width=\textwidth]{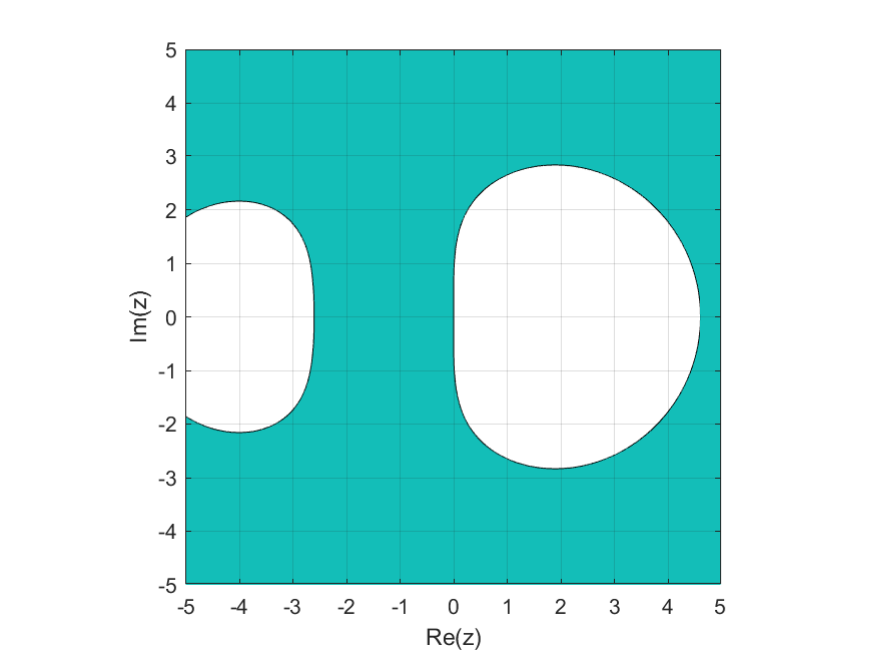}}
\centerline{$\beta_1=-\frac13,\beta_2=1$}
\end{minipage}

\caption{The green parts show the region of absolute stability of the new schemes with $\beta_1=-\frac13$.}
\label{stable_ex1}
\end{figure}

\begin{figure}[htp]
\centering
\begin{minipage}{0.32\linewidth}
\centerline{\includegraphics[width=\textwidth]{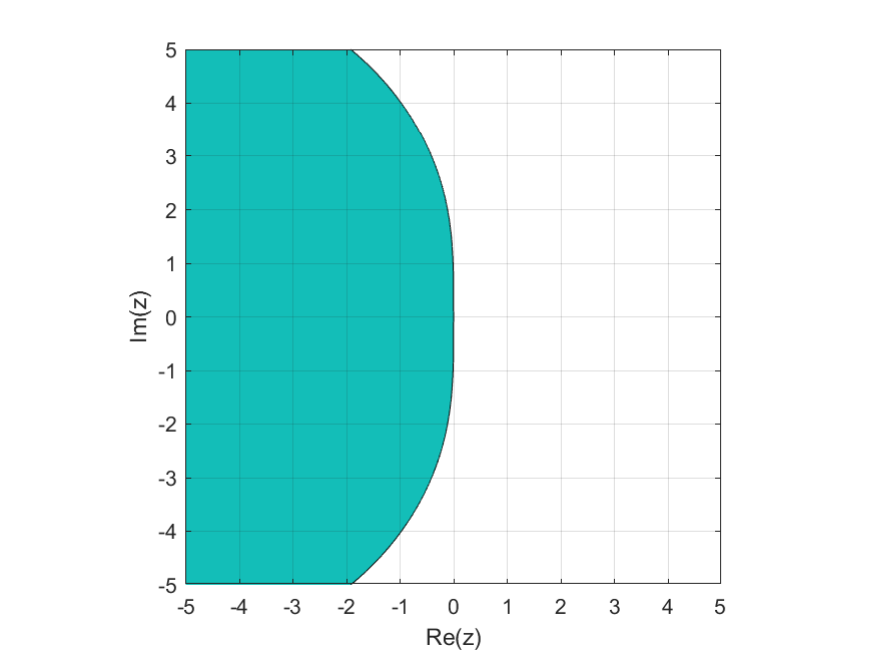}}
\centerline{$\beta_1=\frac13,\beta_2=\frac12$}
\end{minipage}
\begin{minipage}{0.32\linewidth}
\centerline{\includegraphics[width=\textwidth]{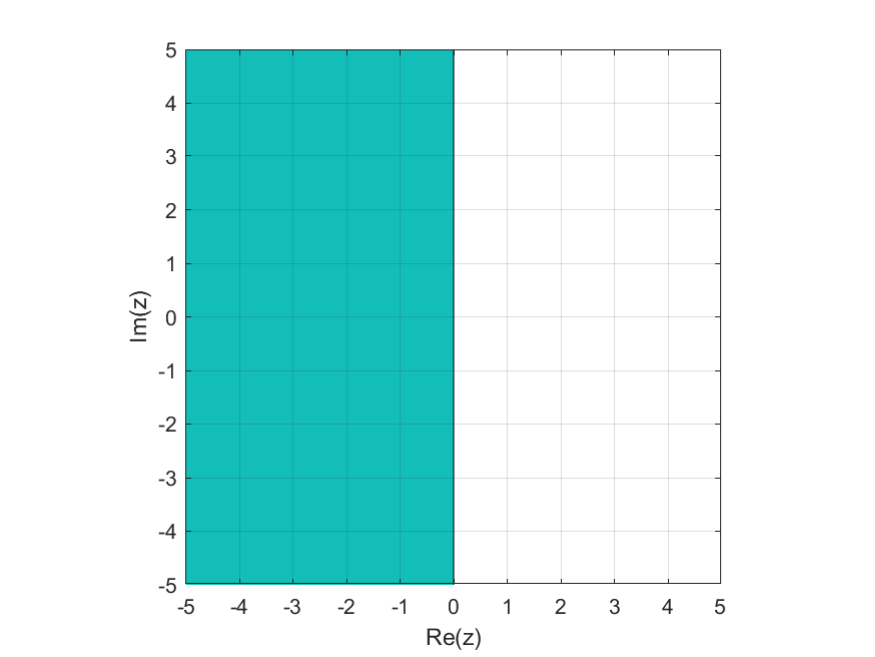}}
\centerline{$\beta_1=\frac13,\beta_2=\frac23$}
\end{minipage}
\begin{minipage}{0.32\linewidth}
\centerline{\includegraphics[width=\textwidth]{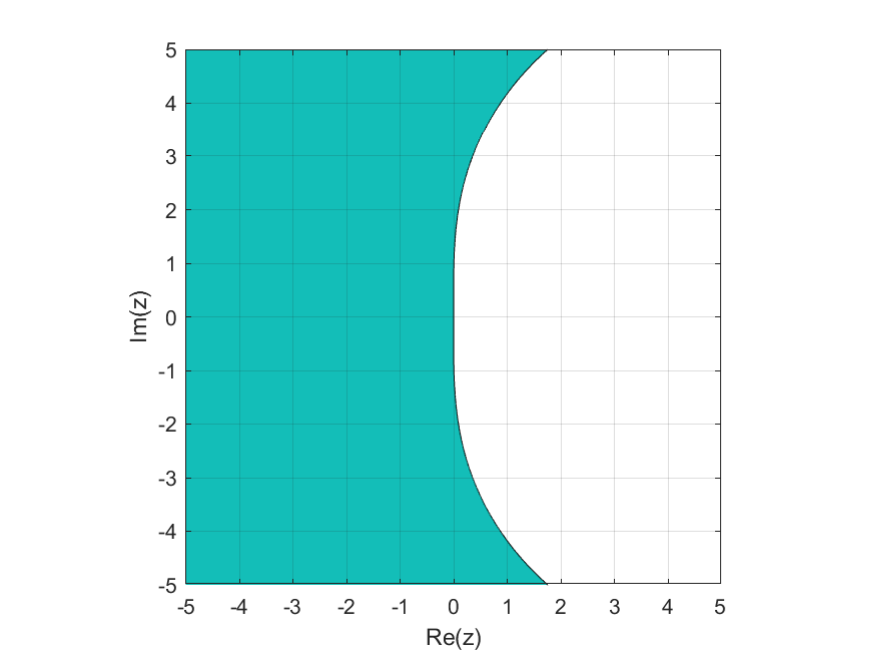}}
\centerline{$\beta_1=\frac13,\beta_2=\frac56$}
\end{minipage}

\caption{The green parts show the region of absolute stability of the new schemes with $\beta_1=\frac13$.}
\label{stable_ex3}
\end{figure}

\begin{figure}[htp]
\centering
\begin{minipage}{0.32\linewidth}
\centerline{\includegraphics[width=\textwidth]{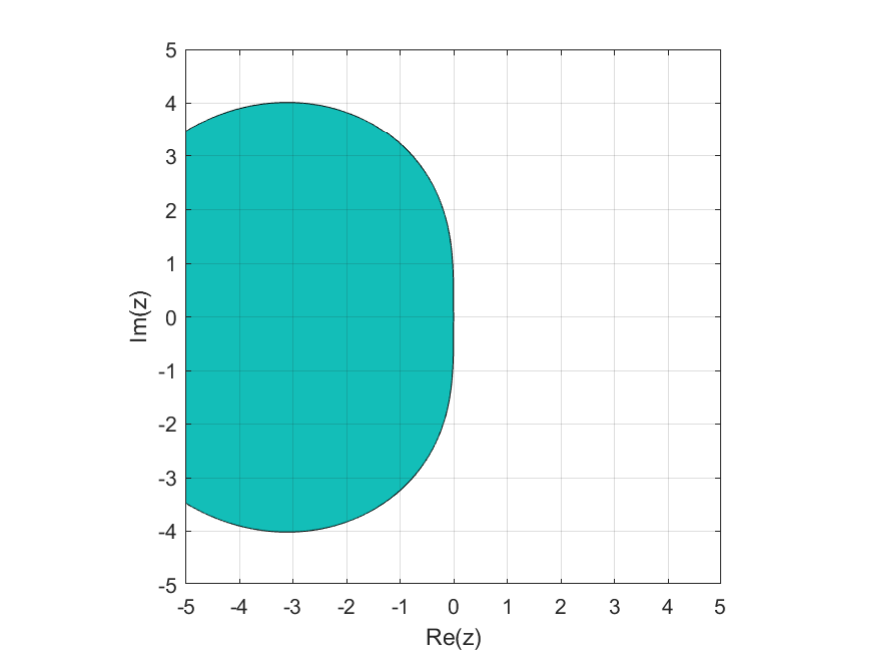}}
\centerline{$\beta_1=\frac12-\frac{\sqrt3}6,\beta_2=1/2$}
\end{minipage}
\begin{minipage}{0.32\linewidth}
\centerline{\includegraphics[width=\textwidth]{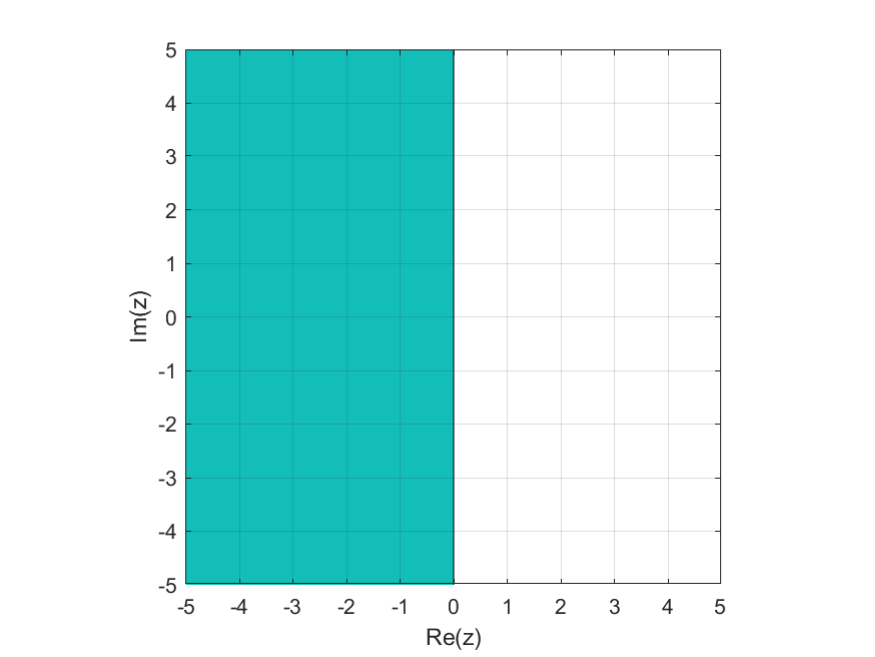}}
\centerline{$\beta_1=\frac12-\frac{\sqrt3}6,\beta_2=\frac12+\frac{\sqrt3}6$}
\end{minipage}
\begin{minipage}{0.32\linewidth}
\centerline{\includegraphics[width=\textwidth]{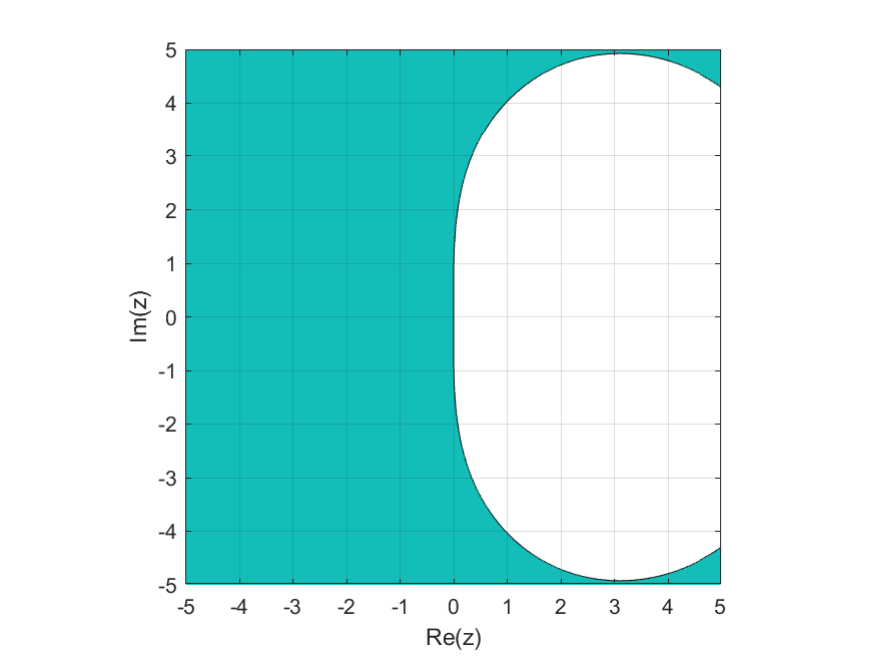}}
\centerline{$\beta_1=\frac12-\frac{\sqrt3}6,\beta_2=\frac12+\frac{\sqrt3}3$}
\end{minipage}

\caption{The green parts show the region of absolute stability of the new schemes with $\beta_1=\frac12-\frac{\sqrt3}6$.}
\label{stable_ex4}
\end{figure}

\subsection{A-stable and L-stable of (\ref{eq5})}
In this subsection, we rigorously establish that the one-step scheme \eqref{eq5} achieves both A-stability and L-stability under the parametric regime where $\beta_1$ and $\beta_2$ satisfy the stipulated algebraic constraints.
\begin{theorem}\label{th3}
If there exists $\beta_1$ and $\beta_2$ to satisfy
\begin{equation}
0\leq\beta_1+\beta_2-1\leq2\beta_1\beta_2,
\end{equation}
then the following inequality always holds true for any given negative real number $z\leq0$:
\begin{equation}
\label{teq1}
\left\vert\frac{g_1(z)}{g_2(z)}\right\vert\leq1.
\end{equation}
\end{theorem}

\begin{proof}
We first demonstrate that the denominator $g_2(z)$ remains strictly non-zero throughout the parameter domain, achieved via proof by contradiction.
Assuming that there exists $z_0<0$ to satisfy $g_2(z_0)=\beta_1\beta_2z_0^2-(\beta_1+\beta_2)z_0+2=0$, then invoking the quadratic formula, we obtain that the roots $z_0$ can be explicitly given by the following expression:
\begin{equation}
z_0=\frac{\beta_1+\beta_2\pm\sqrt{\beta_1^2+\beta_2^2-6\beta_1\beta_2}}{2\beta_1\beta_2}.
\end{equation}

Considering that $\beta_1\beta_2\geq0$ and $\beta_1+\beta_2\geq1$, we have
\begin{equation}
(\beta_1+\beta_2)^2-(\beta_1^2+\beta_2^2-6\beta_1\beta_2)=8\beta_1\beta_2\geq0,
\end{equation}
which means $\beta_1+\beta_2\pm\sqrt{\beta_1^2+\beta_2^2-6\beta_1\beta_2}\geq0$. Then we have $z_0\geq0$ which contradicts the premise $z_0<0$, hence $g_2(z)\neq0$.

Subsequently, we prove inequality \eqref{teq1}. Considering that $g_2(z)\neq0$, then the inequality \eqref{teq1} is equivalent to 
\begin{equation}
\label{teq2}
g_1(z)^2-g_2(z)^2\leq0
\end{equation} 

For notational clarity, we define $h(z)=g_1(z)^2-g_2(z)^2,\;\alpha_1=\beta_1+\beta_2-2\beta_1\beta_2-1$ and $\alpha_2=\beta_1+\beta_2-1$. 
Then $h(z)$ can be derive
\begin{equation}\label{eq14}
\begin{aligned}
h(z)=&\alpha_1\alpha_2z^4+2(\alpha_2^2-\alpha_1)z^3-8\alpha_2z^2+8z \\
=&z(\alpha_2z-2)(\alpha_1z^2+2\alpha_2z-4).
\end{aligned}
\end{equation}

We now just need to demonstrate that 
\begin{equation}
\begin{aligned}
h(z)\leq0,\quad\forall z\leq0,\quad\forall\alpha_1\leq0,\quad\forall\alpha_2\geq0.
\end{aligned}
\end{equation}
Firstly, if $\alpha_1=0$, the expression of $h(z)$ can be reducible to:
\begin{equation}
h(z)=2z(\alpha_2z-2)^2.
\end{equation}
It is now evident that we readily obtain $h(z)\leq0$, $\forall z\leq0$. Secondly, if $\alpha_1<0$ and $\alpha_2=0$, $h(z)$ will be simplified as the following
\begin{equation}
h(z)=-2z(\alpha_1z^2-4).
\end{equation}
It is obviously easy to obtain $h(z)\leq0$, $\forall z\leq0$ and $\forall \alpha_1<0$. Thirdly, if $\alpha_1<0,\;\alpha_2>0$, 
the product of the first two terms of $h(z)$ in \eqref{eq14} satisfies $z(\alpha_2z-2)\geq0$. For the third term of $h(z)$, we have $$(\alpha_1z^2+2\alpha_2z-4)|_z=2\alpha_1z+2\alpha_2\geq0,\quad\forall z\leq0,$$ 
which means $\alpha_1z^2+2\alpha_2z-4$ is increasing for $z\in(-\infty,0]$. Noting that $\alpha_1z^2+2\alpha_2z-4|_{z=0}=-4<0$, we have $\alpha_1z^2+2\alpha_2z-4<0,\ \forall z\leq0$. In summary, we eventually reach $h(z)\leq0$, $\forall z\leq0$.
\end{proof}

\begin{theorem}
For all complex numbers $Re(z)\leq0$, if their exists $\beta_1$ and $\beta_2$ to satisfy
\begin{equation}
\label{teq3}
0\leq\beta_1+\beta_2-1\leq2\beta_1\beta_2,
\end{equation}
then the following inequality always holds true
\begin{equation}
\label{teq4}
\left\vert\frac{g_1(z)}{g_2(z)}\right\vert\leq1,
\end{equation}
which indicates that the new one-step schemes (\ref{eq5}) are A-stable.
\end{theorem}

\begin{proof}
From Theorem \ref{th3}, we are easy to obtain $g_2(z)\neq0$, $\forall Re(z)\leq0$. Then the inequality \eqref{teq4} is equivalent with the following
\begin{equation}
\label{teq5}
h(z)=g_1(z)g_1(\overline z)-g_2(z)g_2(\overline z)\leq0, \quad\forall Re(z)\leq0.
\end{equation}

Define $z=re^{i\theta}$, where $r\geq0$ and $\theta\in[\frac{\pi}{2},\frac{3\pi}{2}]$. We can derive
\begin{equation}\label{eq15}
\begin{aligned}
h(z)=&\alpha_1\alpha_2r^4+2(\alpha_2^2-\alpha_1)r^3\cos\theta-4\alpha_2r^2(1+\cos2\theta)+8r\cos\theta \\
=&\alpha_1\alpha_2r^4(\cos^2\theta+\sin^2\theta)^2+2(\alpha_2^2-\alpha_1)r^3\cos\theta(\cos^2\theta+\sin^2\theta) \\
&-8\alpha_2r^2\cos^2\theta+8r\cos\theta \\
=&\alpha_1\alpha_2(r\cos\theta)^4+2(\alpha_2^2-\alpha_1)(r\cos\theta)^3-8\alpha_2(r\cos\theta)^2+8(r\cos\theta) \\
&+\alpha_1\alpha_2(2\cos^2\theta\sin^2\theta+\sin^4\theta)r^4+2(\alpha_2^2-\alpha_1)\cos\theta\sin^2\theta r^3
\end{aligned}
\end{equation}

Based on the conclusion of Theorem \ref{th3} combined with $Re(z)=r\cos\theta\leq0$, we obtain that the sum of the first four terms in the expression of $h(z)$ in \eqref{eq15} is negative, i.e.,
\begin{equation}
\alpha_1\alpha_2(r\cos\theta)^4+2(\alpha_2^2-\alpha_1)(r\cos\theta)^3-8\alpha_2(r\cos\theta)^2+8(r\cos\theta)\leq0.
\end{equation}

Given $r\geq0$, $\alpha_1\leq0,\;\alpha_2\geq0$ and $\theta\in[\frac{\pi}{2},\frac{3\pi}{2}]$, it is obviously to obtain the fifth and sixth terms of $h(z)$ in \eqref{eq15} are both negative:
\begin{equation}
\begin{aligned}
\alpha_1\alpha_2(2\cos^2\theta\sin^2\theta+\sin^4\theta)r^4&\leq0, \\
2(\alpha_2^2-\alpha_1)\cos\theta\sin^2\theta r^3&\leq0.
\end{aligned}
\end{equation}

Therefore, we conclude $h(z)\le0$, $\forall Re(z)\leq0$.
\end{proof}
\begin{remark}
The existence condition of $\beta_1$ and $\beta_2$ in \eqref{teq3} to keep the proposed one-step schemes \eqref{eq5} A-stable is fairly relaxed. For example, if we fix $\beta_1=1$, then the inequality \eqref{teq3} will become:
\begin{equation}
0\leq\beta_2\leq2\beta_2,
\end{equation}
which implies that it holds for all $\beta_2\geq0$.
\end{remark}

Next, we will determine the conditions on $\beta_1$ and $\beta_2$ that guarantee L-stability of the proposed one-step schemes \eqref{eq5}.
\begin{theorem}
If $\beta_1$ and $\beta_2$ satisfy the following
\begin{equation}
\label{teq6}
\beta_1+\beta_2-1>0,\;(\beta_1-1)(\beta_2-1)=0,
\end{equation}
then $\forall z\leq0$, we obtain
\begin{equation}
\label{teq7}
\lim_{z\to-\infty}\frac{g_1(z)}{g_2(z)}=0.
\end{equation}
which indicates that the new one-step schemes \eqref{eq5} are L-stable.
\end{theorem}

\begin{proof}
Noting that 
$$
(\beta_1-1)(\beta_2-1)=\beta_1\beta_2-(\beta_1+\beta_2-1)=0,
$$
which means $\beta_1\beta_2>0$ and $2\beta_1\beta_2\geq\beta_1+\beta_2-1>0$, then we obtain the new schemes (\ref{eq5}) are also A-stable under the condition of \eqref{teq6}. Meanwhile, $g_1(z)$ will be simplified as $g_1(z)=(2-\beta_1-\beta_2)z+2$ which leads to the following
\begin{equation}
\lim_{z\to-\infty}\frac{g_1(z)}{g_2(z)}=\lim_{z\to-\infty}\frac{(2-\beta_1-\beta_2)z+2}{\beta_1\beta_2z^2-(\beta_1+\beta_2)z+2}=0.
\end{equation}
Then we obtain the new one-step schemes (\ref{eq5}) are L-stable which completes the proof.
\end{proof}
\begin{remark}
The condition \eqref{teq6} for achieving L-stable of the schemes \eqref{eq5} can be equivalently expressed as the following inequality:
\begin{equation*}
\beta_1=1,~\beta_2>0,\quad or\quad \beta_2=1,~\beta_1>0, 
\end{equation*}
which implies that if either $\beta_1$ or $\beta_2$ takes the value 1, it suffices for the other ($\beta_2$ or $\beta_1$) to be positive.
\end{remark}

In Fig. \ref{stable_ex2}, we plot the stability regions of both RK2 and the new proposed scheme \eqref{eq5} by using the same parameters $\beta_1=0,\;\beta_2=1$. Theorem \ref{thm1} indicates that with such parameter selections, our proposed scheme \eqref{eq5} is a stabilized modification of the RK2 method by adding stabilized term. Fig. \ref{stable_ex2} further demonstrates that the stability region of RK2 is significantly smaller than that of our newly proposed scheme \eqref{eq5}, which fully confirms the theoretical analysis. Besides, we also plotted the stability region by setting $\beta_1=\frac23,\;\beta_2=1$, which the scheme \eqref{eq5} is L-stable in this case.
\begin{figure}[htp]
\centering
\begin{minipage}{0.32\linewidth}
\centerline{\includegraphics[width=\textwidth]{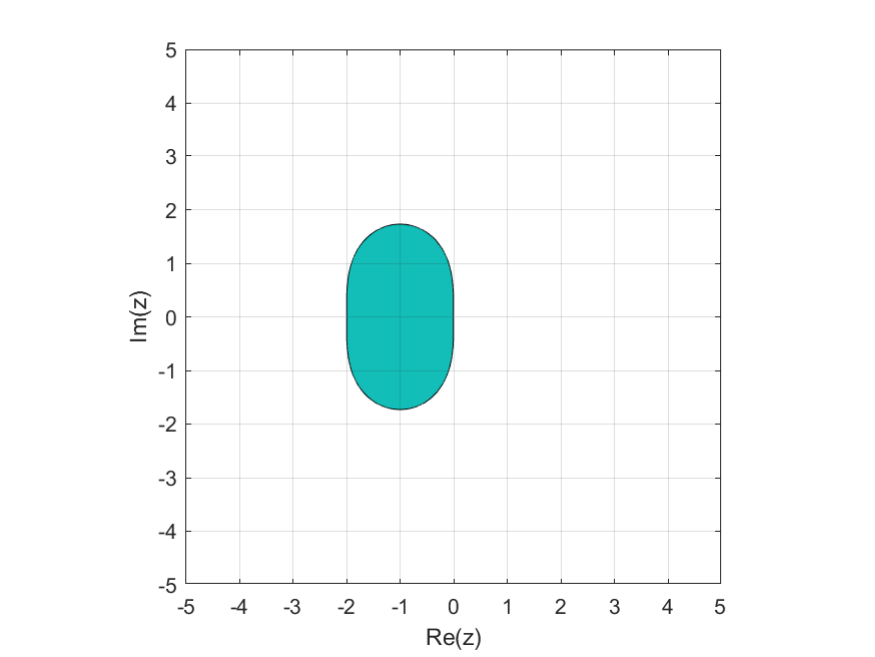}}
\centerline{RK2}
\end{minipage}
\begin{minipage}{0.32\linewidth}
\centerline{\includegraphics[width=\textwidth]{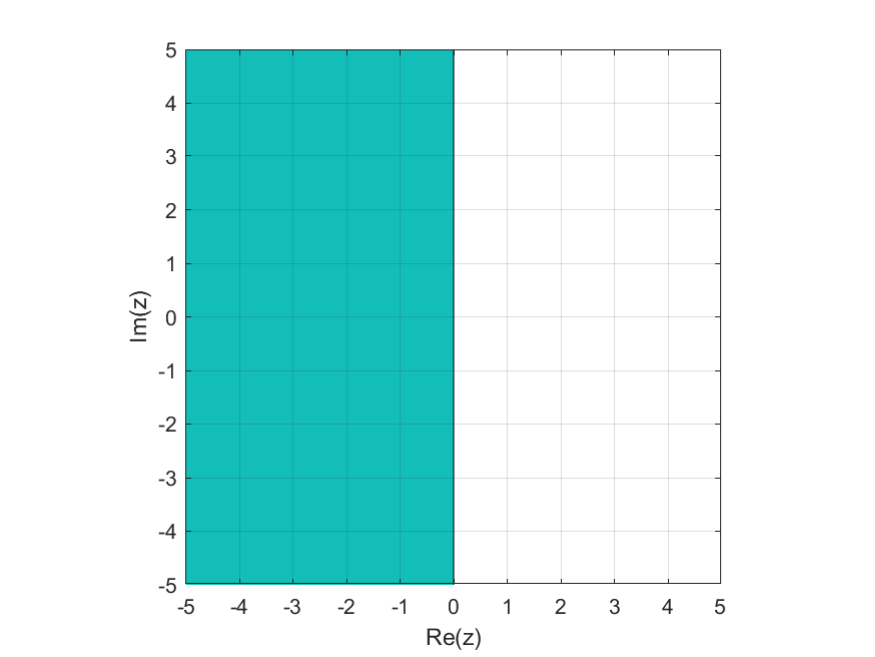}}
\centerline{$\beta_1=0,\;\beta_2=1$}
\end{minipage}
\begin{minipage}{0.32\linewidth}
\centerline{\includegraphics[width=\textwidth]{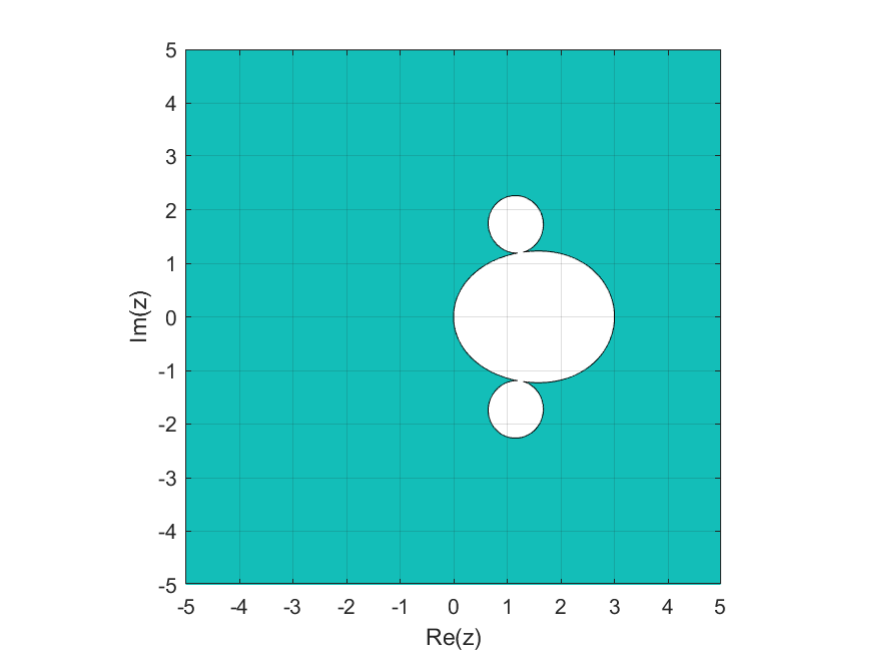}}
\centerline{$\beta_1=\frac23,\;\beta_2=1$}
\end{minipage}

\caption{The green parts show the region of absolute stability.}
\label{stable_ex2}
\end{figure}
\subsection{Linear stability regions of \eqref{eq13}}
In this subsection, we analyze the stability regions of the second high-order one-step scheme \eqref{eq13}. Considering the multiplicity of involved parameters, we choose to fix the value of $\beta_3$ to streamline the analysis.

For the test equation $u_t=\lambda u$, the proposed scheme \eqref{eq13} reduces to the following
\begin{equation}
\begin{aligned}
&\left[(-\beta_1\lambda\Delta t+1)(-\beta_2\lambda\Delta t+1)(-\beta_3\lambda\Delta t+1)-(\beta_1+\beta_2+\beta_3)\lambda\Delta t+5\right]u^{n+1}\\
=&[-(\beta_1-1)\lambda\Delta t+1][-(\beta_2-1)\lambda\Delta t+1][-(\beta_3-1)\lambda\Delta t+1]u^n \\
&+[-(\beta_1+\beta_2+\beta_3-3)\lambda\Delta t+5]u^n.
\end{aligned}
\end{equation}

In Fig.~\ref{stable_3order_ex1} and \ref{stable_3order_ex2}, we plot the stability regions of the new schemes for different $\beta_1$, $\beta_2$ and $\beta_3$. It can be found that for any fixed $\beta_1$ and $\beta_2$, an A-stable one-step scheme can be systematically achieved through deliberate adjustment of $\beta_3$.

\begin{figure}[htp]
\centering
\begin{minipage}{0.32\linewidth}
\centerline{\includegraphics[width=\textwidth]{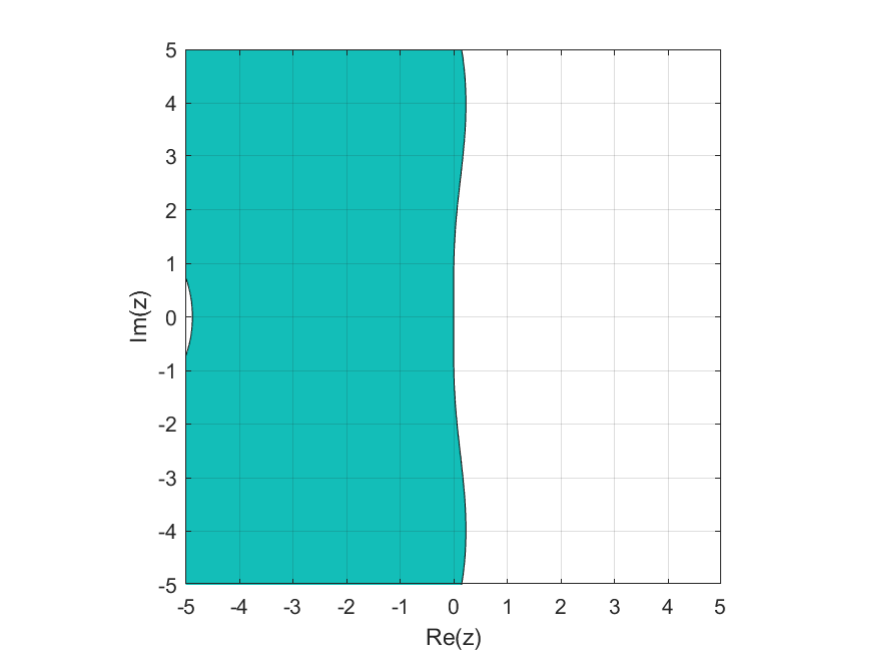}}
\centerline{$\beta_3=\frac34$}
\end{minipage}
\begin{minipage}{0.32\linewidth}
\centerline{\includegraphics[width=\textwidth]{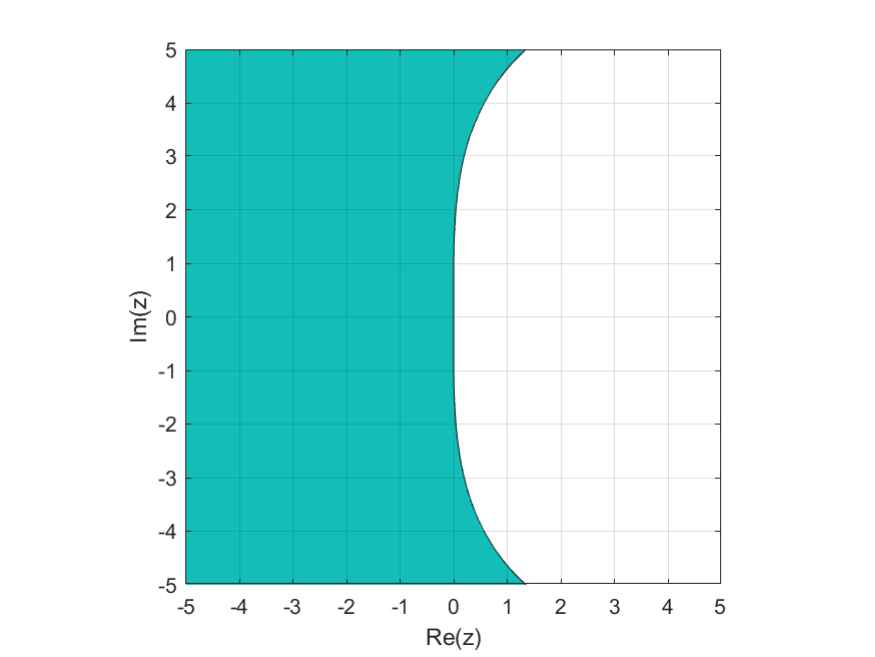}}
\centerline{$\beta_3=1$}
\end{minipage}
\begin{minipage}{0.32\linewidth}
\centerline{\includegraphics[width=\textwidth]{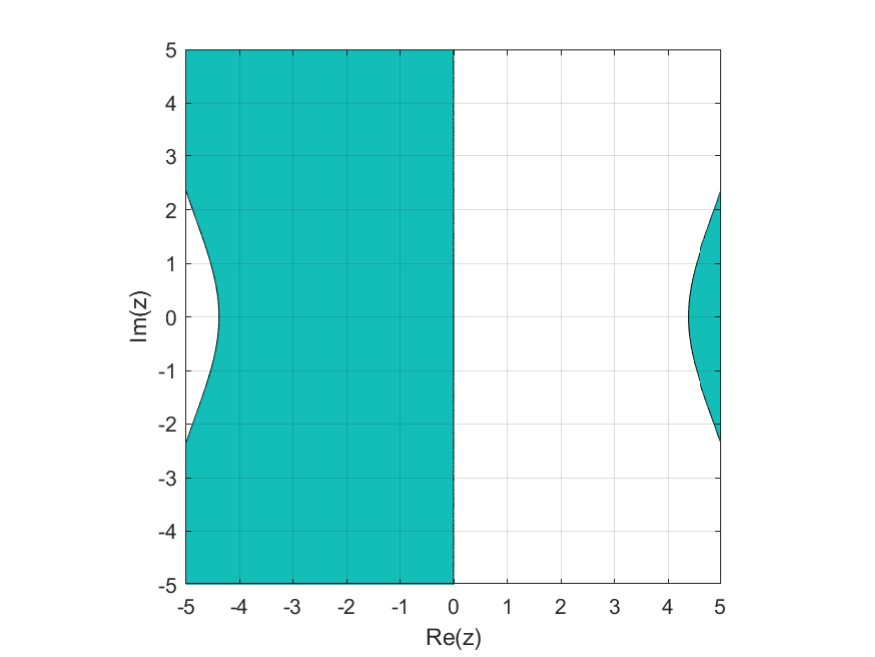}}
\centerline{$\beta_3=\frac54$}
\end{minipage}

\caption{The green parts show the region of absolute stability of the new schemes with $\beta_1=-\frac14,\;\beta_2=\frac12$.}
\label{stable_3order_ex1}
\end{figure}

\begin{figure}[htp]
\centering
\begin{minipage}{0.32\linewidth}
\centerline{\includegraphics[width=\textwidth]{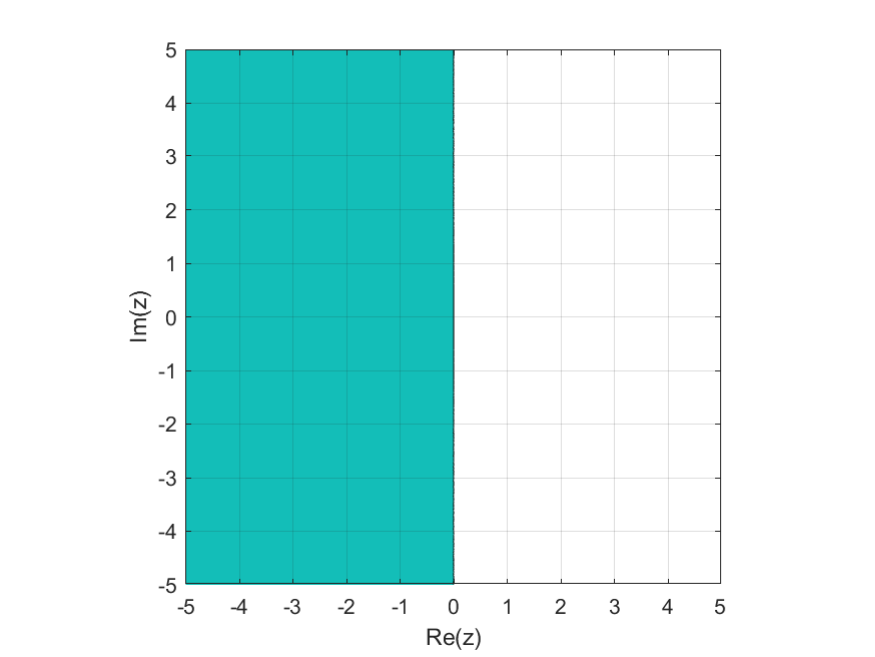}}
\centerline{$\beta_3=\frac34$}
\end{minipage}
\begin{minipage}{0.32\linewidth}
\centerline{\includegraphics[width=\textwidth]{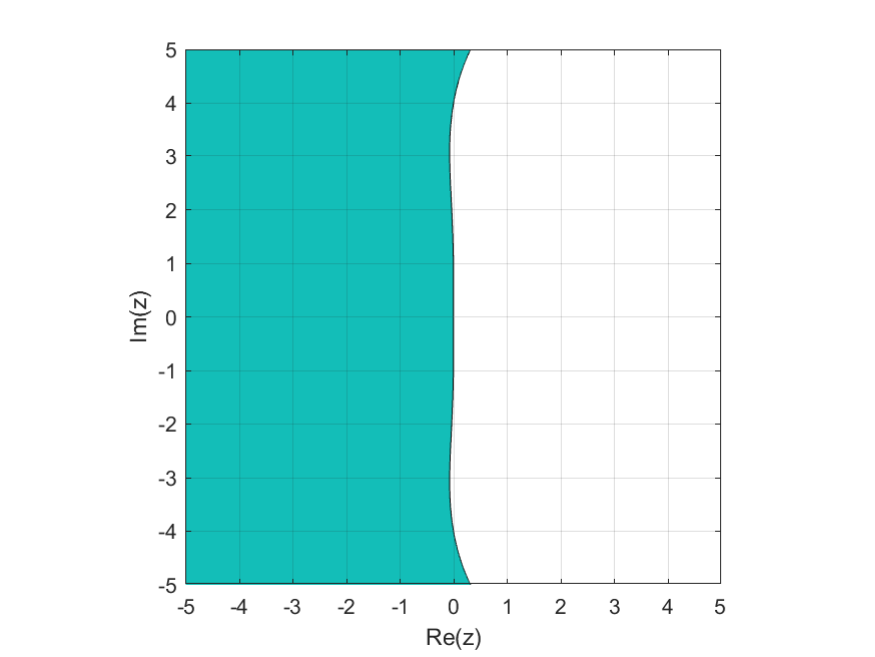}}
\centerline{$\beta_3=1$}
\end{minipage}
\begin{minipage}{0.32\linewidth}
\centerline{\includegraphics[width=\textwidth]{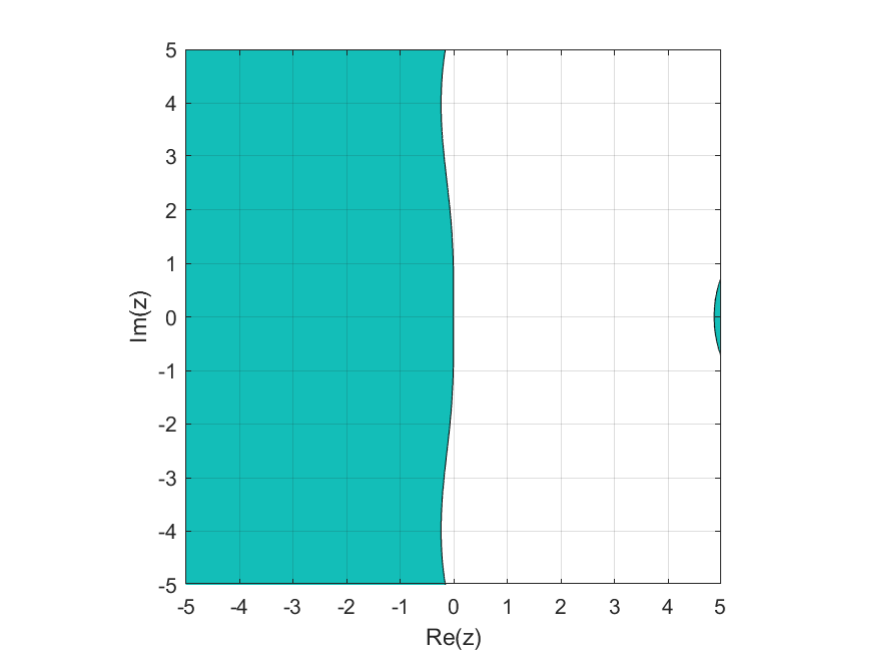}}
\centerline{$\beta_3=\frac54$}
\end{minipage}

\caption{The green parts show the region of absolute stability of the new schemes with $\beta_1=\frac14,\;\beta_2=\frac12$.}
\label{stable_3order_ex2}
\end{figure}

In terms of the experimental results on the stability domain, selecting appropriate parameters $\beta_1$, $\beta_2$, and $\beta_3$ will have a significant impact on stability. Next, we will establish proofs of A-stable under certain conditions. For the sake of narrative convenience, we set $z=\lambda\Delta$ and 
\begin{equation}
\begin{aligned}
g_1(z)&=((1-\beta_1)z+1)((1-\beta_2)z+1)((1-\beta_2)z+1)-(\beta_1+\beta_2+\beta_3-3)z+5, \\
g_2(z)&=(1-\beta_1z)(1-\beta_2z)(1-\beta_3z)-(\beta_1+\beta_2+\beta_3)z+5.\\
\end{aligned}
\end{equation}

\begin{theorem}\label{th6}
If there exists three distinct numbers $\beta_1,\;\beta_2$ and $\beta_3$ to satisfy
\begin{equation}\label{con1}
\beta_3=1,\;2(\beta_1+\beta_2)\geq1+6\beta_1\beta_2>0,
\end{equation}
then the following inequality always holds true for all complex numbers $Re(z)\leq0$:
\begin{equation}
\label{teq461}
\left\vert\frac{g_1(z)}{g_2(z)}\right\vert\leq1.
\end{equation}
which indicates that the new one-step scheme \eqref{eq13} is A-stable.
\end{theorem}

\begin{proof}
%
%
%
The inequality \eqref{teq461} is equivalent with the following
\begin{equation}
\label{teq462}
h(z)=g_1(z)g_1(\bar z)-g_2(z)g_2(\bar z)\leq0.
\end{equation}

Define $z=re^{i\theta}$, where $r\geq0$ and $\theta\in[\frac{\pi}{2},\frac{3\pi}{2}]$. We can derive
\begin{equation}
\begin{aligned}
h(z)
=&r^3\cos^3\theta(\beta_1\beta_2r^2\cos^2\theta+(1-2\beta_1-2\beta_2)r\cos\theta+6)(-\beta_1\beta_2r\cos\theta+1) \\
&+2r\cos\theta(\beta_1\beta_2(r\cos\theta)^2+(1-2\beta_1-2\beta_2)r\cos\theta+6)^2 \\
&+8r^6(\cos\theta-1)(\cos\theta+1)\cos\theta(-\beta_1^2+\beta_1\beta_2+\beta_1-\beta_2^2+\beta_2-1) \\
&+r^5(\cos\theta-1)(\cos\theta+1)(1+\cos^2\theta)(2\beta_1+2\beta_2-6\beta_1\beta_2-1) \\
&+8\beta_1\beta_2r^5\cos^2\theta(\cos\theta-1)(\cos\theta+1)(1+\beta_1+\beta_2) \\
&-2\beta_1\beta_2r^4\cos\theta(\cos\theta-1)(\cos\theta+1)(1+\cos^2\theta)(\beta_1+\beta_2+\beta_1\beta_2) \\
&+r^3\beta_1^2\beta_2^2(\cos\theta-1)(\cos\theta+1)(\cos\theta^4+\cos^2\theta+1).
\end{aligned}
\end{equation}

When $2(\beta_1+\beta_2)\geq1+6\beta_1\beta_2,\;r\geq0$ and $\theta\in[\frac{\pi}{2},\frac{3\pi}{2}]$, it is obviously to obtain
\begin{equation}
\begin{aligned}
\beta_1\beta_2r^2\cos^2\theta+(1-2\beta_1-2\beta_2)r\cos\theta+6\geq0, \\
-\beta_1\beta_2r\cos\theta+1\geq0, \\
(\cos\theta-1)(\cos\theta+1)\leq0, \\
-\beta_1^2+\beta_1\beta_2+\beta_1-\beta_2^2+\beta_2-1\leq0. \\
\end{aligned}
\end{equation}
By regarding $\beta_1\beta_2r^2\cos^2\theta+(1-2\beta_1-2\beta_2)r\cos\theta+6$ and $-\beta_1^2+\beta_1\beta_2+\beta_1-\beta_2^2+\beta_2-1$ in the above expressions as quadratic functions of $r\cos\theta$ and $\beta_1$ respectively, we can readily obtain the desired conclusion. Therefore, we conclude $h(z)\leq0,\;\forall Re(z)\leq0,\;g_2(z)\neq0$.
\end{proof}

We proceed to conduct parametric studies with varying $\beta_i$ $(i=1,2,3)$ values to quantify their influence on stability regions. Fig.~\ref{stable_3order_ex3} shows stability regions for RK3, the new proposed scheme \eqref{eq13} with $\beta_1=0,\;\beta_2=\frac12,\;\beta_3=1$ and $\beta_1=0.3,\;\beta_2=2,\;\beta_3=1$. It can be observed that when $\beta_1$, $\beta_2$ and $\beta_3$ satisfy the derived A-stable condition \eqref{con1}, the stability region undergoes significant enlargement compared to other cases.
\begin{figure}[htp]
\centering
\begin{minipage}{0.32\linewidth}
\centerline{\includegraphics[width=\textwidth]{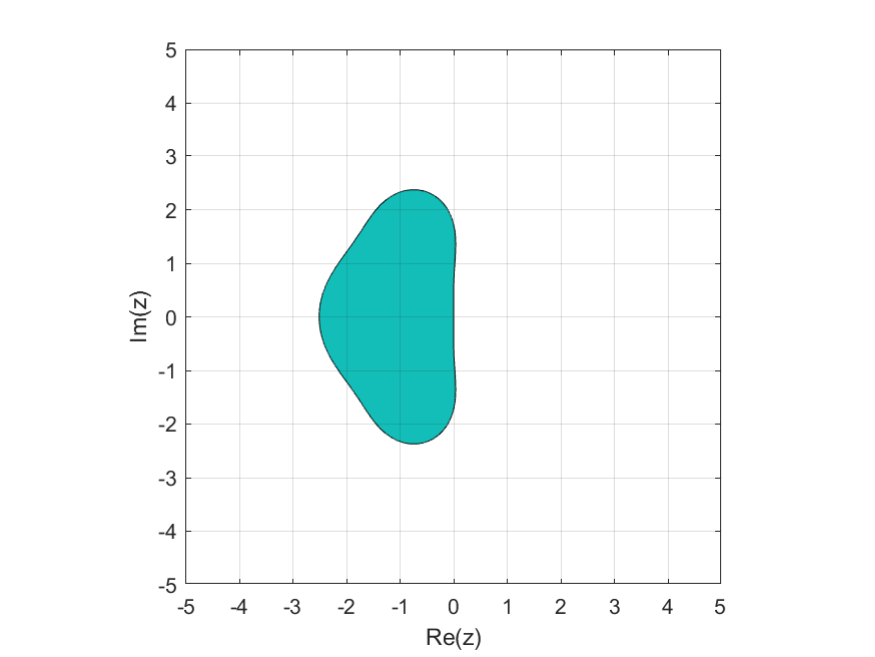}}
\centerline{RK3}
\end{minipage}
\begin{minipage}{0.32\linewidth}
\centerline{\includegraphics[width=\textwidth]{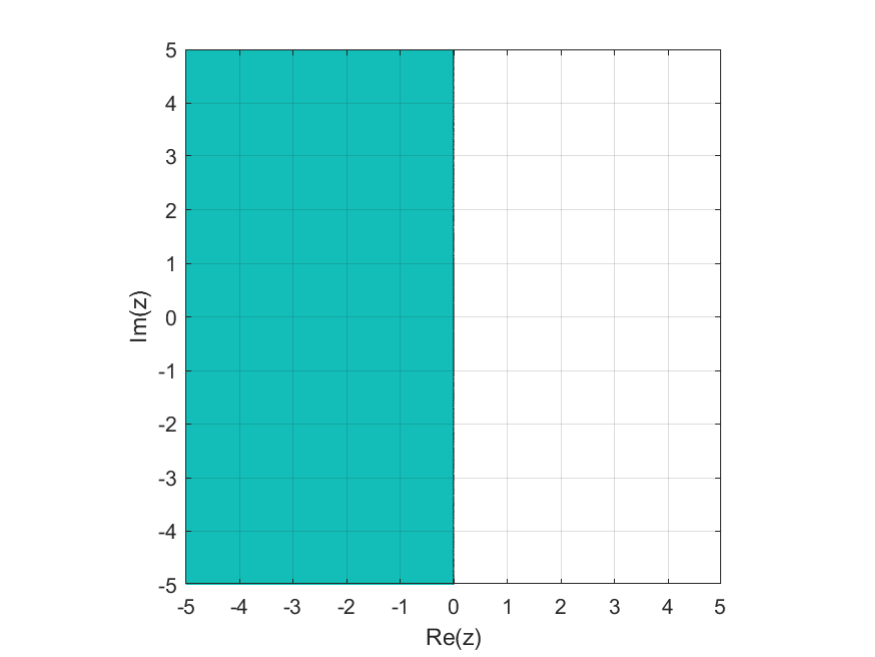}}
\centerline{$\beta_1=0,\;\beta_2=\frac12,\;\beta_3=1$}
\end{minipage}
\begin{minipage}{0.32\linewidth}
\centerline{\includegraphics[width=\textwidth]{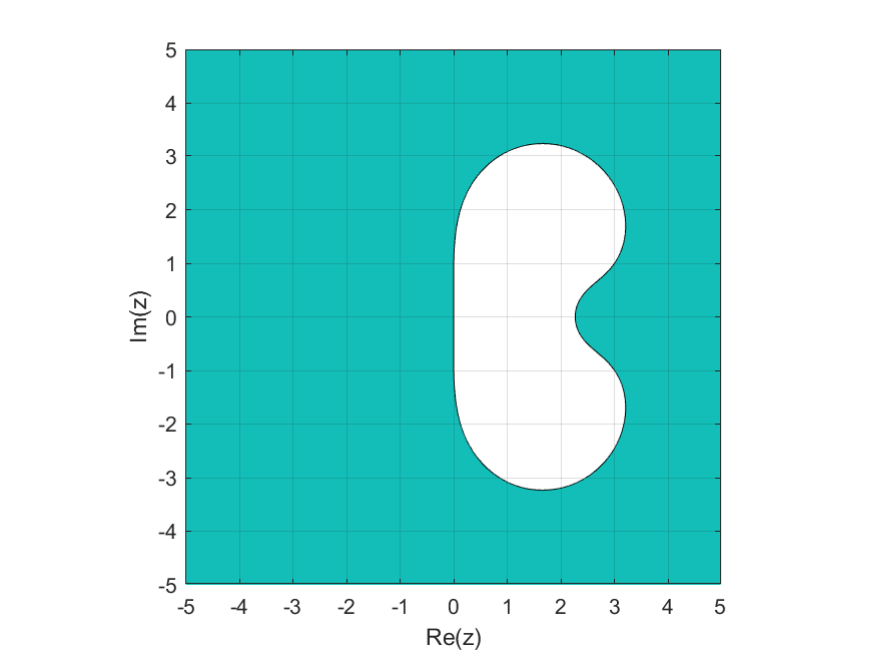}}
\centerline{$\beta_1=0.3,\;\beta_2=2,\;\beta_3=1$}
\end{minipage}

\caption{The green parts show the region of absolute stability.}
\label{stable_3order_ex3}
\end{figure}

\section{Numerical examples}
In this section, we report the numerical results which are obtained from the implementation of the new proposed one-step schemes. We first verify the stability and temporal convergence rates with a smooth initial data for some ordinary and partial differential equations. Next we provide some numerical approximations of some classical linear models such as convection dominated diffusion equation and nonlinear models such as Allen-Cahn equation to validate our theoretical results.

\textbf{Example.1.} We validate the stability and convergence order of the new one-step schemes for a simple system of ordinary differential equations
\begin{equation}
\begin{aligned}
&u_t(t)+\lambda u(t)=0,\;t\in(0,1], \\
&u(0)=1.
\end{aligned}
\end{equation}

We can readily calculate the exact solution to satisfy $u(t)=e^{-\lambda t}$ and observe that this solution exhibits sharp variation in a neighborhood of the origin. To investigate the dependence of stability on parameters $\beta_i$ in our proposed one-step method, numerical experiments were conducted with several different parameters $\lambda=1000$, $10000$, $100000$ with $\Delta t=\frac{1}{32}$. In Fig.~\ref{ex1} , we plot the exact solution and the numerical solution obtained by the one-step scheme \eqref{eq5} with different $\beta_1$ and $\beta_2$. When $\beta_1=\frac49$, $\beta_2=2$ and $\beta_1=\frac{3-\sqrt3}{6}$, $\beta_2=\frac{3+\sqrt3}{6}$, the scheme \eqref{eq5} satisfies A-stable but not L-stable. When $\beta_1=0.1$ and $\beta_2=1$, the scheme \eqref{eq5} satisfies L-stable. We observe that as the parameter $\lambda$ increases, the L-stable scheme converges at an extremely fast rate, whereas the other two schemes either converge very slowly or exhibit significant fluctuations before convergence. Although the two schemes that only satisfy A-stable have a higher order of convergence than the L-stable scheme, their performance is unsatisfactory based on the results.
\begin{figure}[htp]
\centering
\begin{minipage}{0.32\linewidth}
\centerline{\includegraphics[width=\textwidth]{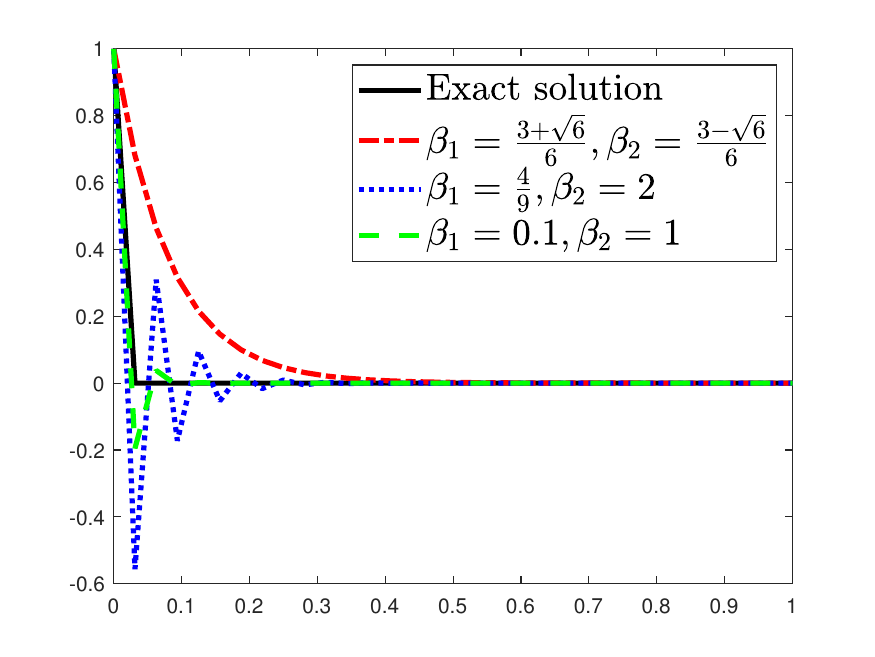}}
\centerline{$\lambda=1000$}
\end{minipage}
\begin{minipage}{0.32\linewidth}
\centerline{\includegraphics[width=\textwidth]{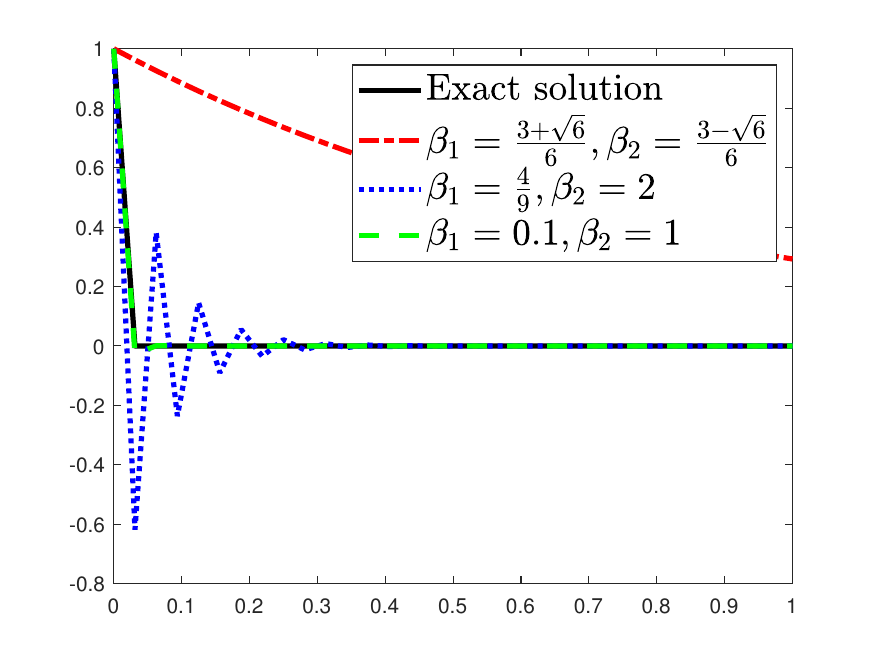}}
\centerline{$\lambda=10000$}
\end{minipage}
\begin{minipage}{0.32\linewidth}
\centerline{\includegraphics[width=\textwidth]{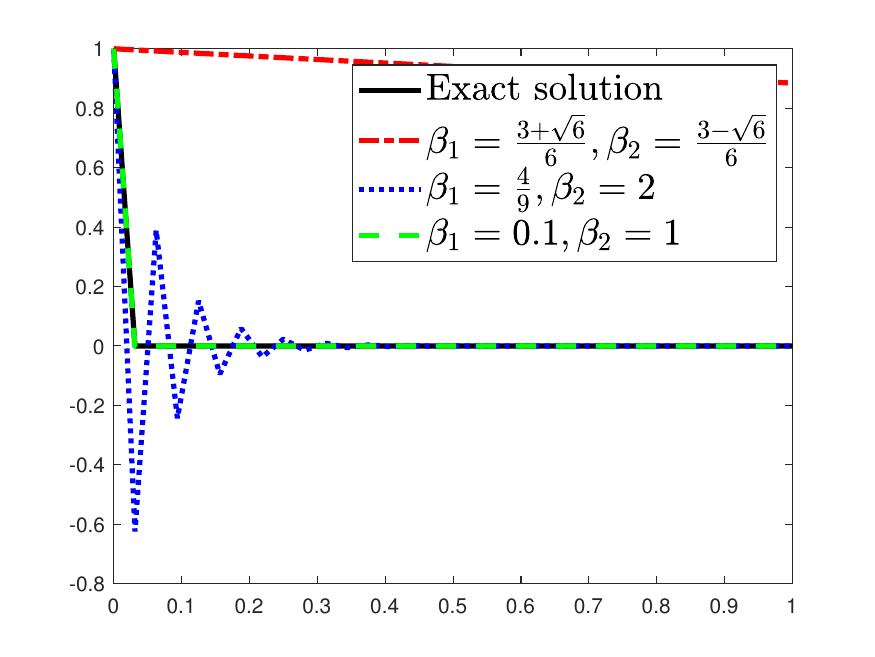}}
\centerline{$\lambda=100000$}
\end{minipage}

\caption{The exact solution and the numerical solution obtained by the second-order scheme with different $\beta_1$, $\beta_2$ and $\lambda$.}
\label{ex1}
\end{figure}

\textbf{Example.2.} We give the following example to test the convergence rates of the proposed schemes \eqref{eq5} and \eqref{eq13}. Consider the following partial differential equation:
\begin{equation}
\begin{aligned}
&u_t(x,y,t)-\Delta u(x,y,t)=f(x,y,t),\;t\in(0,1), \\
&u(x,y,0)=0,
\end{aligned}
\end{equation}
where $(x,y)\in[-\pi,\pi]\times[-\pi,\pi]$. We choose a forcing function such that the exact solution is $u(x,y,t)=\sin x\sin y\sin t$. We use the Fourier Galerkin method with $N_x=N_y=128$ in space so that the spatial discretization error is negligible compared to the time discretization error. 
We plot the $L^2$ errors between the numerical solution and the exact solution at $T = 1$ with different time step sizes in Fig.~\ref{ex2_1} by using the second-order schemes \eqref{eq5} with different $\beta_1$ and $\beta_2$. One can observe that our numerical scheme \eqref{eq5} is asymptotically second-order accurate in time for all variables. 

\begin{figure}[htp]
\centering
\begin{minipage}{0.32\linewidth}
\centerline{\includegraphics[width=\textwidth]{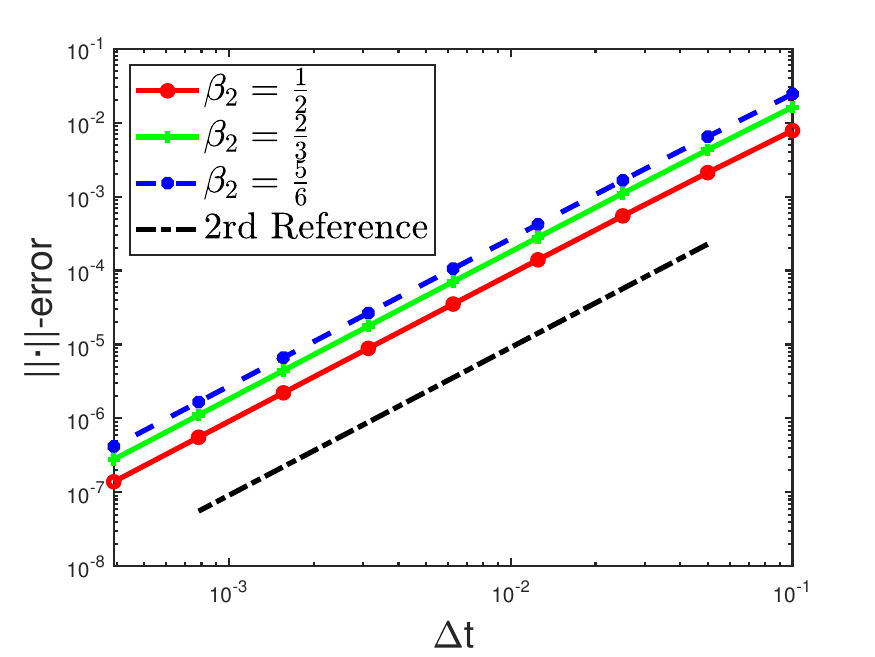}}
\centerline{$\beta_1=1$}
\end{minipage}
\begin{minipage}{0.32\linewidth}
\centerline{\includegraphics[width=\textwidth]{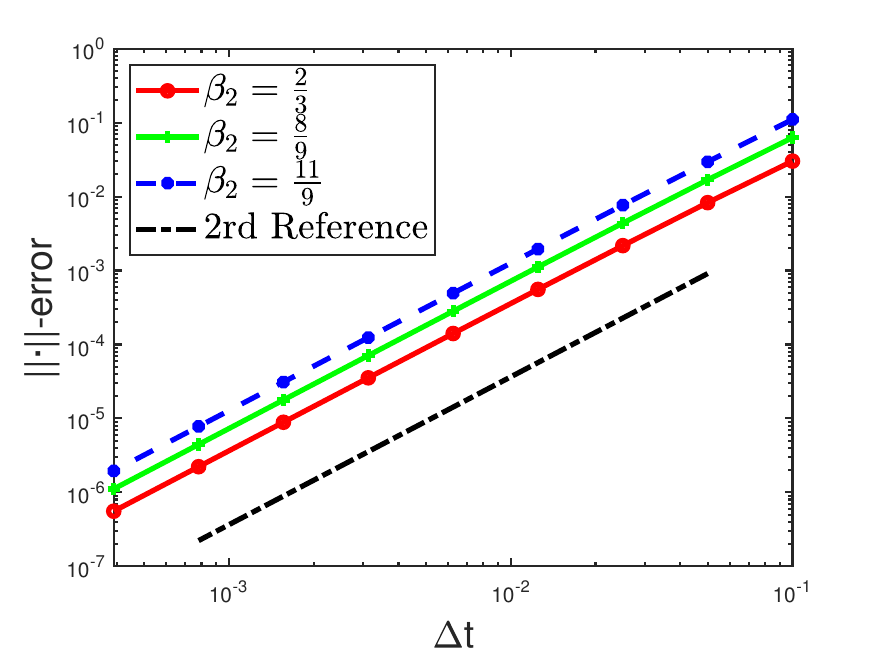}}
\centerline{$\beta_1=2$}
\end{minipage}
\begin{minipage}{0.32\linewidth}
\centerline{\includegraphics[width=\textwidth]{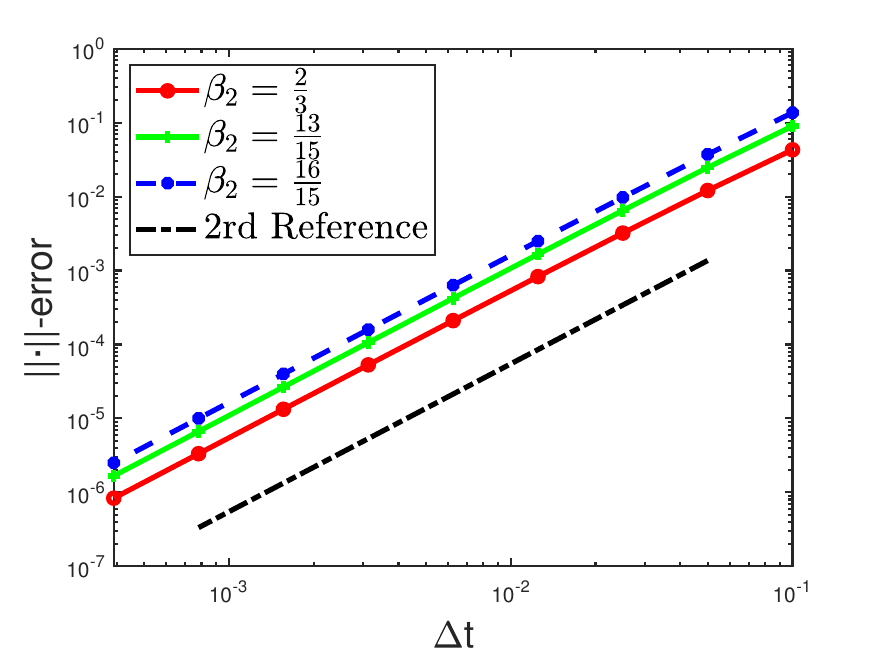}}
\centerline{$\beta_1=3$}
\end{minipage}

\caption{Temporal convergence rates for scheme \eqref{eq5} with different $\beta_1$ and $\beta_2$.}
\label{ex2_1}
\end{figure}

With $\beta_1$ and $\beta_2$ selected to satisfy $3(2\beta_1-1)(2\beta_2-1)=-1$, Theorem \ref{thm2} shows that it will improve the convergence order to 3 under this condition. The first figure in Fig.~\ref{ex2_2} demonstrates that the convergence order of scheme \eqref{eq5} reaches indeed the third order when $\beta_1$ and $\beta_2$ conform to the prescribed constraint. Besides, the second figure in Fig.~\ref{ex2_2} shows that scheme \eqref{eq5} will be the fourth order when $\beta_1=\frac{3+\sqrt{3}}{6}$ and $\beta_2=\frac{3-\sqrt{3}}{6}$ which is consistent with the theoretical results in Theorem \ref{thm2}.

\begin{figure}[htp]
\centering
\begin{minipage}{0.32\linewidth}
\centerline{\includegraphics[width=\textwidth]{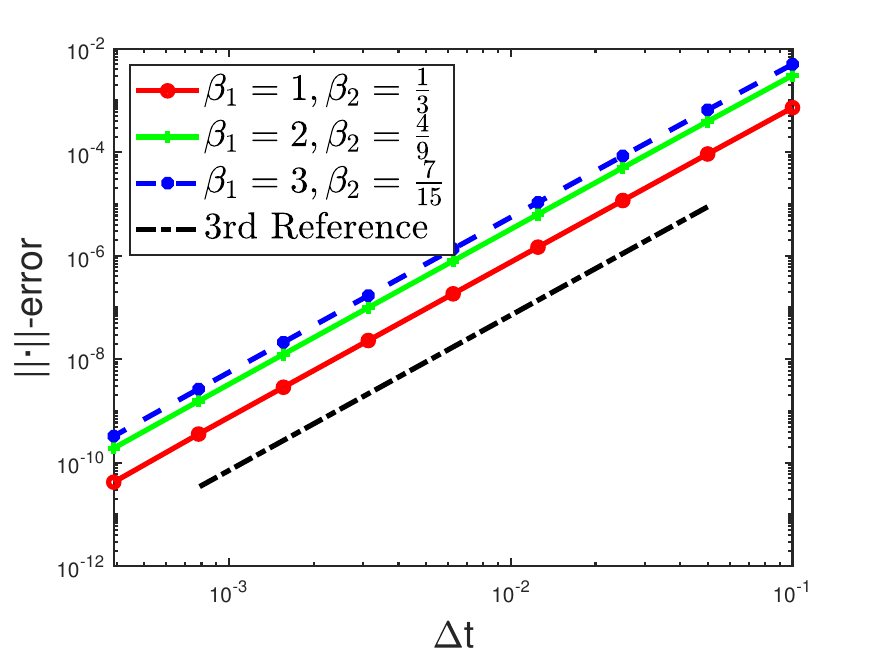}}
\centerline{third order}
\end{minipage}
\begin{minipage}{0.32\linewidth}
\centerline{\includegraphics[width=\textwidth]{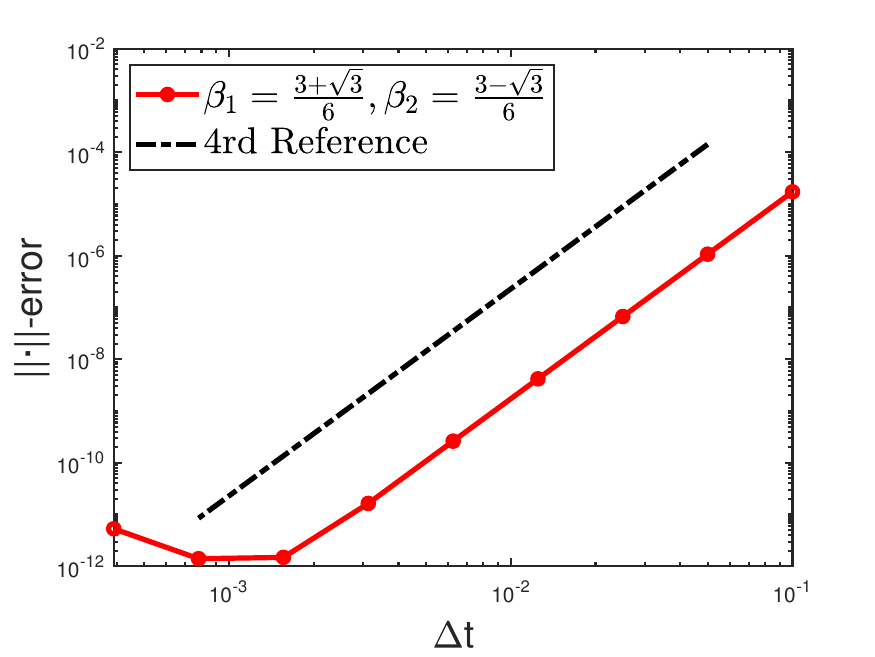}}
\centerline{fourth order}
\end{minipage}

\caption{The superconvergence of the proposed one-step scheme \eqref{eq5} when $\beta_1$ and $\beta_2$ satisfy the prescribed constraint.}
\label{ex2_2}
\end{figure}

Next, we continue to use Example 2 to test the third-order convergence of the second one-step scheme \eqref{eq13}. As demonstrated in  Fig.~\ref{ex2_3}, the A-stable parameter selection of $\beta_i$ elevates the scheme to third-order convergence. Meanwhile, optimized parameter values of $\beta_1$, $\beta_2$ and $\beta_3$ elevate the order to fourth- and sixth-order.

\begin{figure}[htp]
\centering
\begin{minipage}{0.32\linewidth}
\centerline{\includegraphics[width=\textwidth]{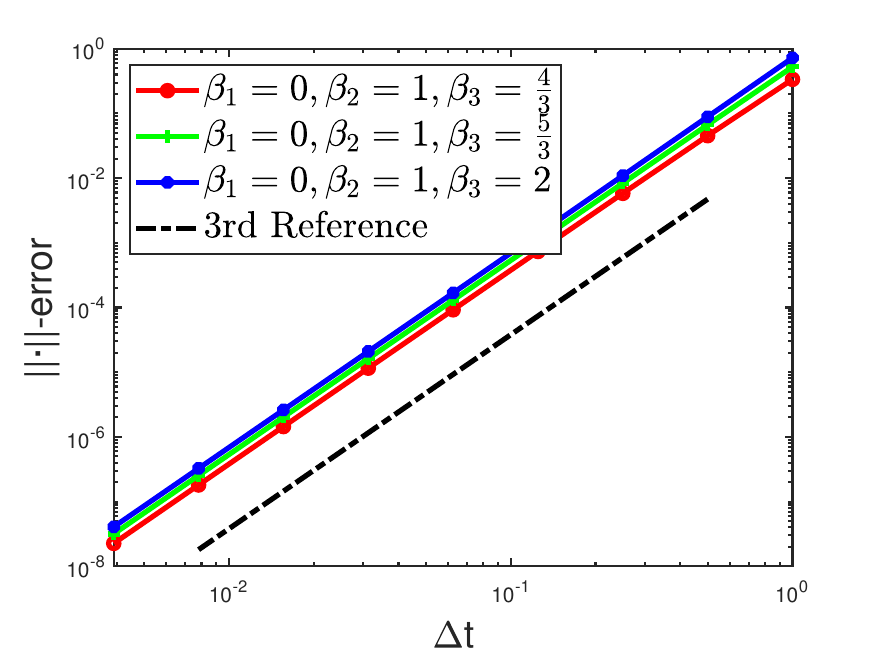}}
\centerline{third order}
\end{minipage}
\begin{minipage}{0.32\linewidth}
\centerline{\includegraphics[width=\textwidth]{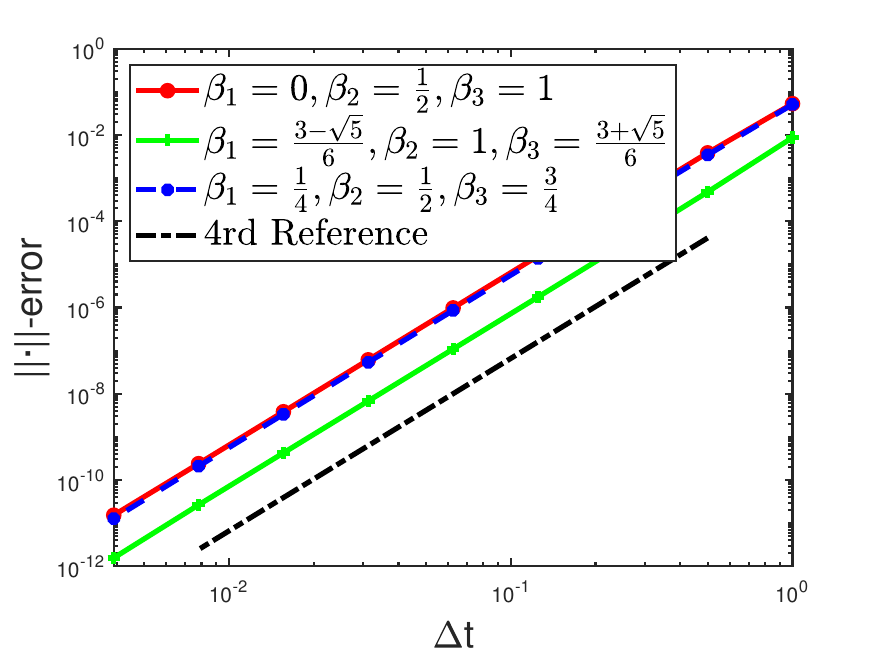}}
\centerline{fourth order}
\end{minipage}
\begin{minipage}{0.32\linewidth}
\centerline{\includegraphics[width=\textwidth]{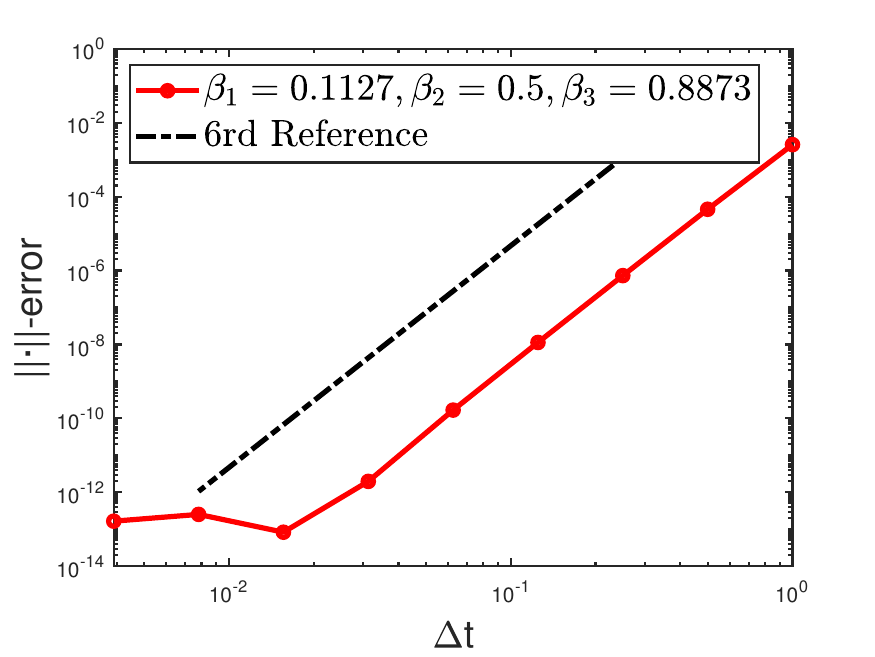}}
\centerline{sixth order}
\end{minipage}

\caption{The convergence and superconvergence of the proposed one-step scheme \eqref{eq13} when $\beta_1$, $\beta_2$ and $\beta_3$ satisfy the prescribed constraint.}
\label{ex2_3}
\end{figure}

\textbf{Example.3.} In this experiment we consider the proposed one-step scheme \eqref{eq5} to solve the following convection dominated diffusion problems
\begin{equation*}
c_t+\mathbf u\cdot\nabla c-\nabla\cdot(\mathbf K\nabla c)=0,\quad (x,y,t)in\Omega\times(0,T],
\end{equation*}
with periodic boundary conditions in the square domain $\Omega=[0,1]\times[0,1]$ with the velocity field of $\mathbf u=(u_x,u_y)=(\sin(\pi x),\sin(\pi y))$ satisfies $\nabla\cdot\mathbf u=0$ and $\mathbf K=\diag(\mathbf{K}_x,\mathbf{K}_y)$ is the diffusion tensor. We take the initial data to be
\begin{equation*}
c(x,y,0)=
\left\{
\begin{matrix}
1,&\;0.1\leq x\leq0.3,\;0.1\leq y\leq 0.3, \\
0,&\;otherise,
\end{matrix}
\right.
\end{equation*}
Then numerical experiments are carried out using various diffusion coefficients of $\mathbf K=2\times10^{-3}$ and $\mathbf K=5\times10^{-4}$. We use the Fourier Galerkin method with $N_x=N_y=128$ in space. The proposed one-step scheme with $\beta_1=\frac23,$ and $\beta_2=1$ is performed to attain numerical solution with time steps of $\Delta t=10^{-3}$, while the reference exact solution is evaluated using the fine mesh of $\Delta t=10^{-5}$.
to obtain the reference solution. Fig.~\ref{ex6_1} gives the contour plots of reference exact solution and numerical solution with $\mathbf{K}=2\times10^{-3}$ at $t=0.25$, $0.35$ and $0.5$. We proceed to examine the performance of our proposed scheme using a smaller diffusion coefficient $\mathbf{K}=5\times10^{-4}$ and its associated outcomes are illustrated in Fig.~\ref{ex6_3}. The results in Fig.~\ref{ex6_3} show the efficacy of our proposed L-stable scheme \eqref{eq5} in accurately tracking steep fronts without introducing numerical oscillations, even when convection effect is particularly prominent. All the numerical solutions are almost the same as the reference exact solutions at all the three time, which verifies the effectiveness of the proposed method. We proceed to examine the performance of our proposed scheme using both a larger diffusion coefficient, $\mathbf K = 5\times10^{-3}$, and a smaller one $\mathbf K = 5\times10^{-4}$. The associated outcomes are illustrated in Fig.~\ref{ex6_3}. Observing these figures, it becomes evident that the numerical solutions generated by our scheme exhibit good consistency with reference exact solutions under varying diffusion conditions. All the process observed in Fig.~\ref{ex6_1} and Fig.~\ref{ex6_3} is consistent with the findings reported in \cite{qin2023positivity}.

\begin{figure}[htp]
\centering
\begin{minipage}{0.32\linewidth}
\centerline{\includegraphics[width=\textwidth]{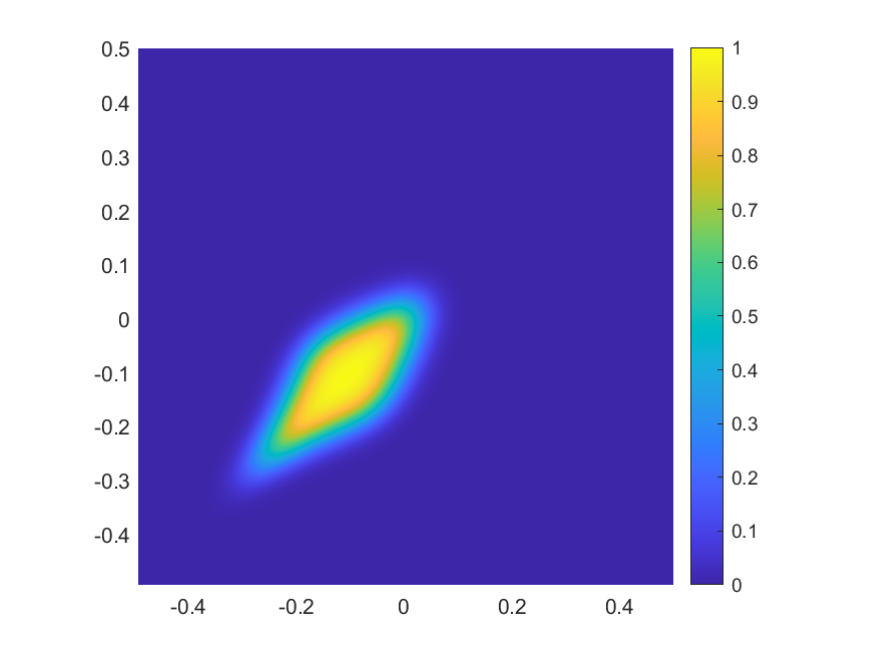}}
\centerline{t=0.25}
\end{minipage}
\begin{minipage}{0.32\linewidth}
\centerline{\includegraphics[width=\textwidth]{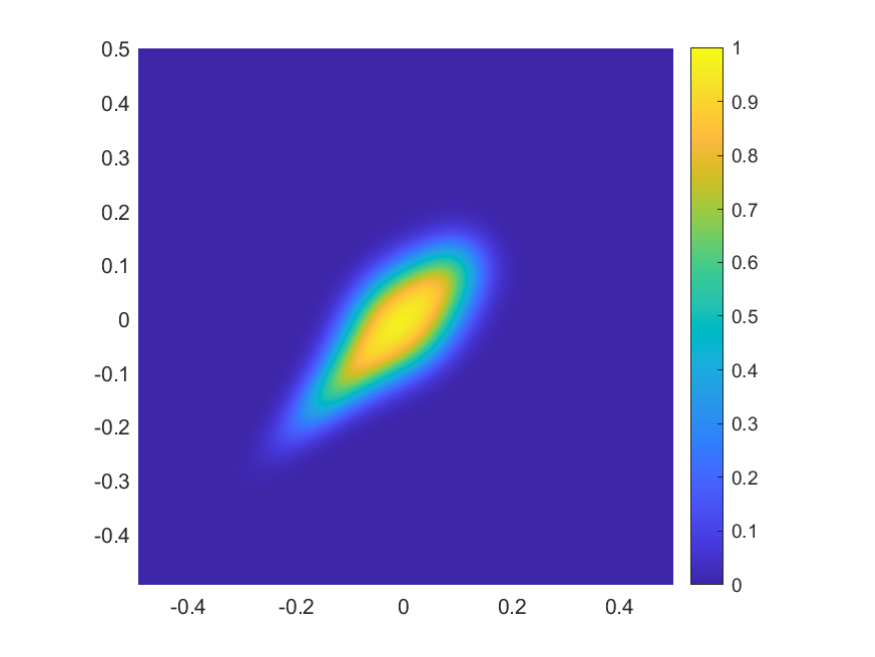}}
\centerline{t=0.35}
\end{minipage}
\begin{minipage}{0.32\linewidth}
\centerline{\includegraphics[width=\textwidth]{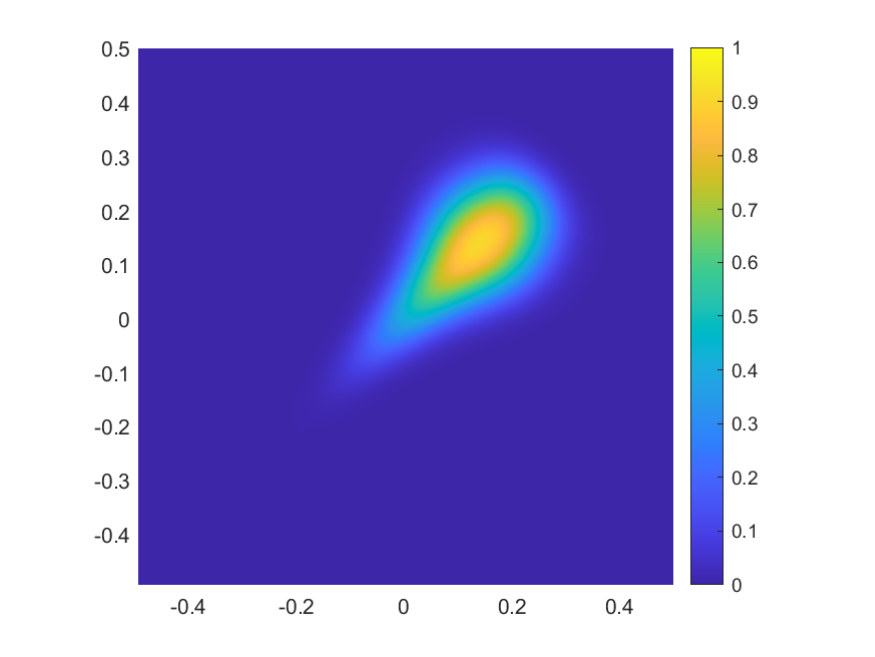}}
\centerline{t=0.5}
\end{minipage}
\centerline{Reference solution}

\begin{minipage}{0.32\linewidth}
\centerline{\includegraphics[width=\textwidth]{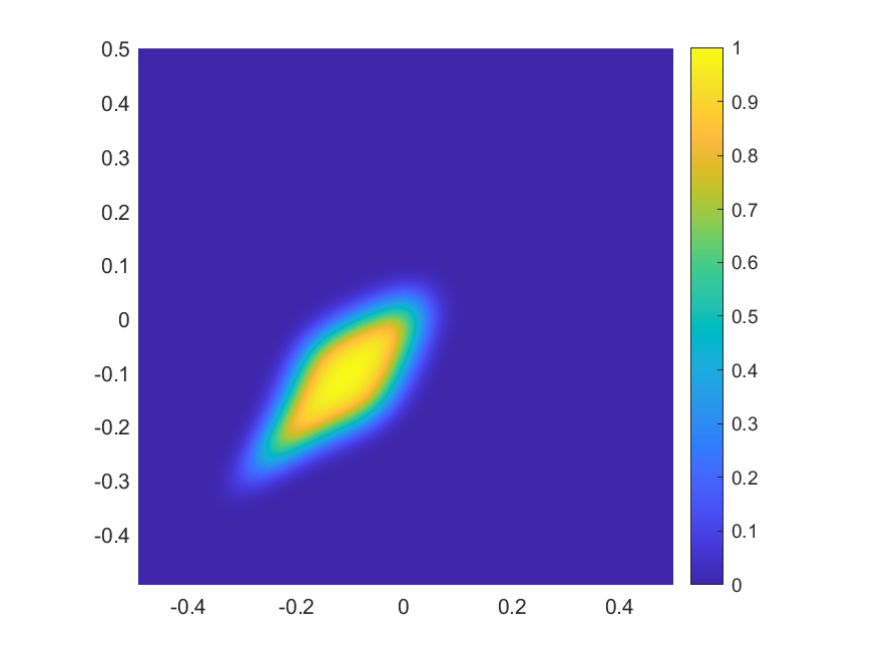}}
\centerline{t=0.25}
\end{minipage}
\begin{minipage}{0.32\linewidth}
\centerline{\includegraphics[width=\textwidth]{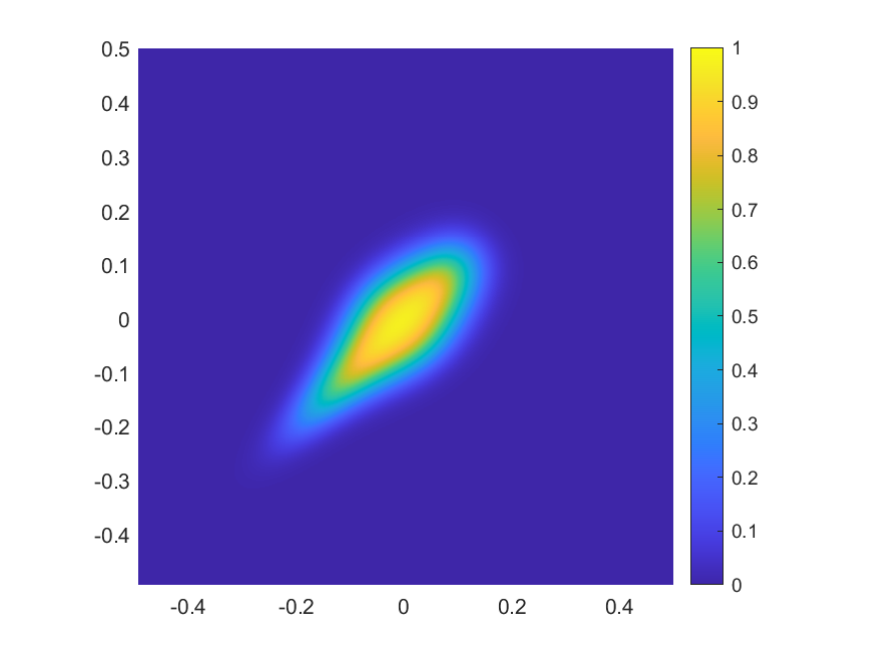}}
\centerline{t=0.35}
\end{minipage}
\begin{minipage}{0.32\linewidth}
\centerline{\includegraphics[width=\textwidth]{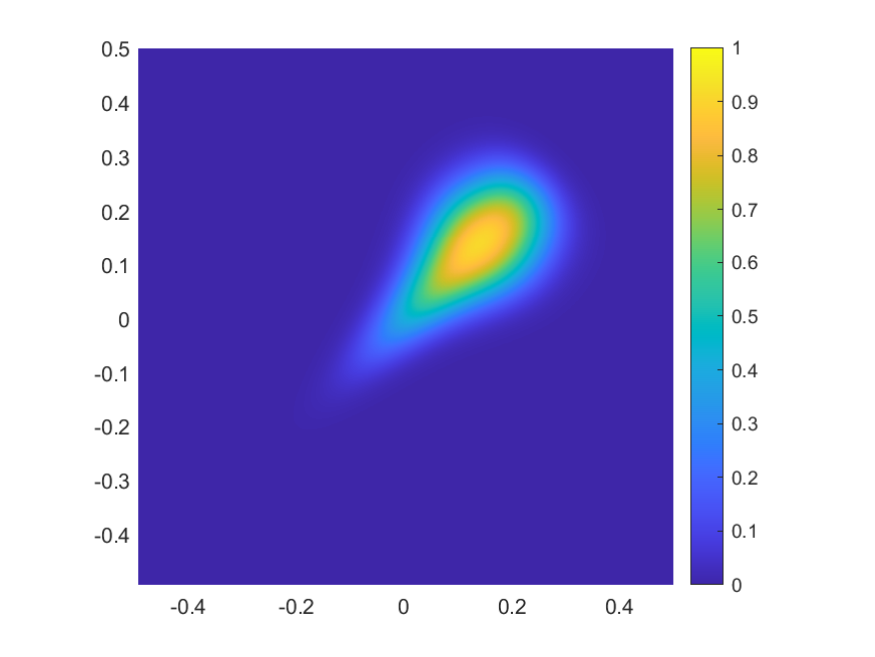}}
\centerline{t=0.5}
\end{minipage}
\centerline{Numerical solution for scheme \eqref{eq5} with $\beta_1=\frac23,\;\beta_2=1$}

\caption{The contour plots of the moving square wave with the diffusion coefficient of $\mathbf K=2\times10^{-3}$.}
\label{ex6_1}
\end{figure}

%
%

\begin{figure}[htp]
\centering
\begin{minipage}{0.32\linewidth}
\centerline{\includegraphics[width=\textwidth]{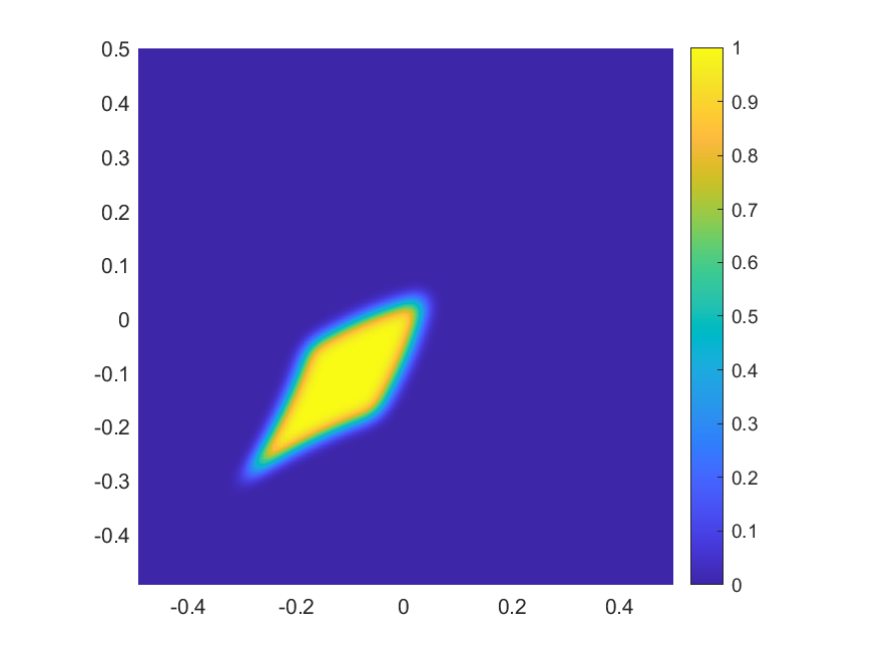}}
\centerline{t=0.25}
\end{minipage}
\begin{minipage}{0.32\linewidth}
\centerline{\includegraphics[width=\textwidth]{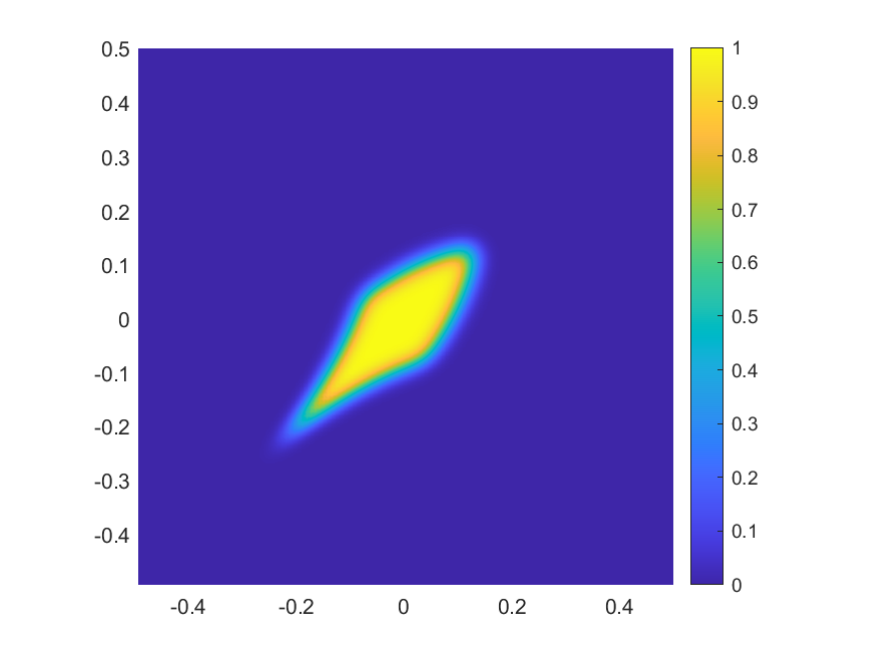}}
\centerline{t=0.35}
\end{minipage}
\begin{minipage}{0.32\linewidth}
\centerline{\includegraphics[width=\textwidth]{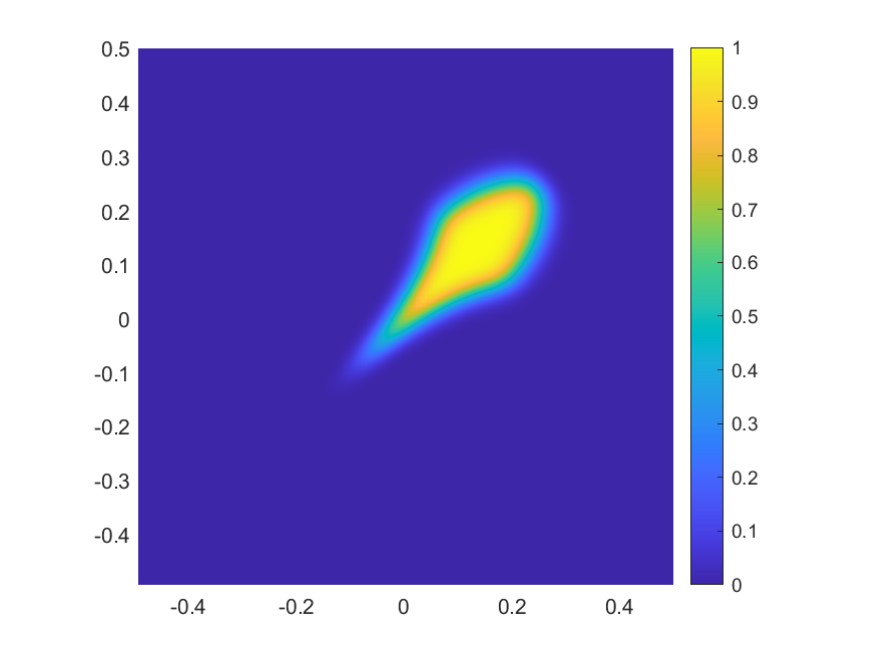}}
\centerline{t=0.5}
\end{minipage}
\centerline{Reference solution}

\begin{minipage}{0.32\linewidth}
\centerline{\includegraphics[width=\textwidth]{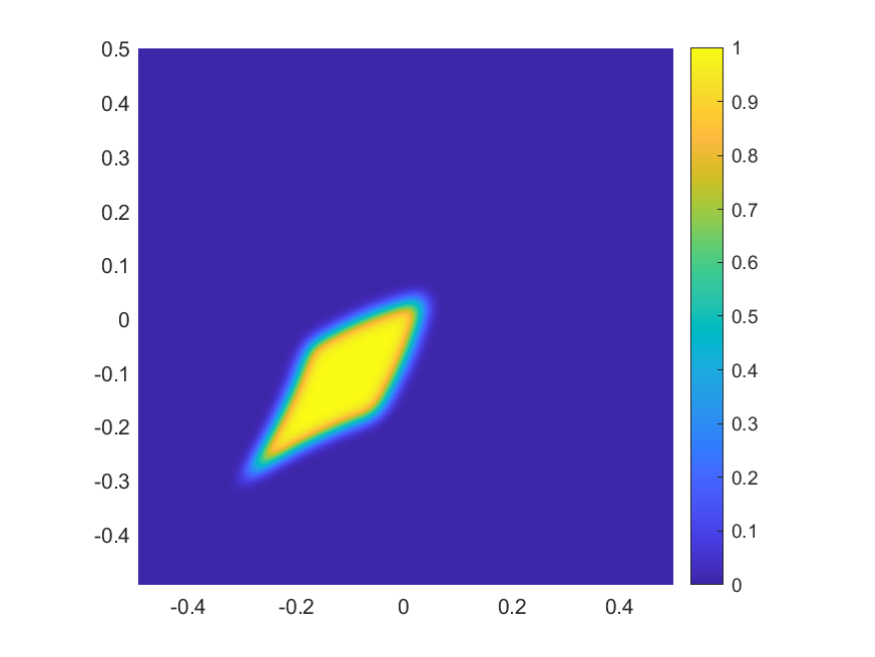}}
\centerline{t=0.25}
\end{minipage}
\begin{minipage}{0.32\linewidth}
\centerline{\includegraphics[width=\textwidth]{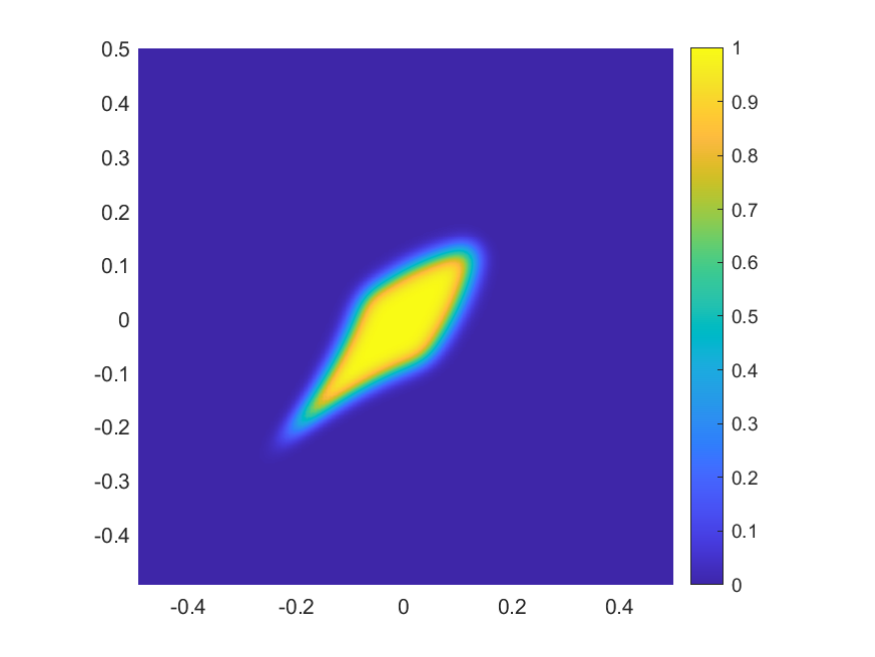}}
\centerline{t=0.35}
\end{minipage}
\begin{minipage}{0.32\linewidth}
\centerline{\includegraphics[width=\textwidth]{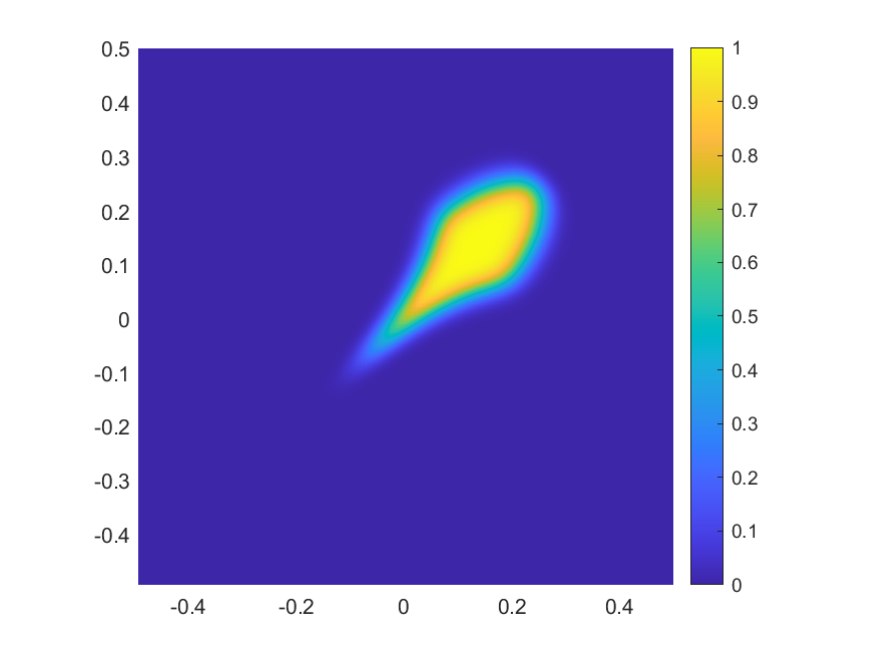}}
\centerline{t=0.5}
\end{minipage}
\centerline{Numerical solution for scheme \eqref{eq5} with $\beta_1=\frac23,\;\beta_2=1$}

\caption{The contour plots of the moving square wave with the diffusion coefficient of $\mathbf K=5\times10^{-4}$.}
\label{ex6_3}
\end{figure}
\textbf{Example.4.} In the following example, we solve a benchmark problem for the Allen-Cahn equation:
\begin{equation}
\begin{aligned}
u_t(x,y,t)&=-\frac{u^3(x,y,t)-u(x,y,t)}{\epsilon^2}+\Delta u(x,y,t),
\end{aligned}
\end{equation}
where $u$ is subject to periodic boundary conditions and $(x,y)\in[-\pi,\pi]\times[-\pi,\pi]$. We replace $f(t^{n+\beta_1})$ with $f(t^{n+\beta_1},\hat u^{n+\beta_1})$ and $f(t^{n+\beta_2})$ with $f(t^{n+\beta_2},\hat u^{n+\beta_2})$ in scheme \eqref{eq5}, where
\begin{equation}
\hat u^{n+\beta_1}=u^n+\beta_1\Delta t(f(t^n,u^n)-\mathcal Lu^n),\;\hat u^{n+\beta_2}=u^n+\beta_2\Delta t(f(t^n,u^n)-\mathcal Lu^n).
\end{equation}
The initial condition is given as
\begin{equation}
u(x,y,0)=\left\{
\begin{array}{lr}
1,\;(\sqrt 3-0.3)^2\leq (x-0)^2+(y-1)^2\leq3, \\
1,\;(\sqrt 3-0.3)^2\leq (x-\frac{\sqrt3}{2})^2+(y+\frac12)^2\leq3, \\
1,\;(\sqrt 3-0.3)^2\leq (x+\frac{\sqrt3}{2})^2+(y+\frac12)^2\leq3, \\
-1,\;others.
\end{array}
\right.
\end{equation}
We use the Fourier Galerkin method with $N_x=N_y=500$ in space. We set $\Delta t=10^{-6}$, $\epsilon=0.2$, $\beta_1=\frac23$ and $\beta_2=1$ and plotted the solution evolution at $t=0$, $0.005$, $0.01$, $0.02$, $0.05$, $0.1$ in Fig.~\ref{ex5_1}. 

\begin{figure}[htp]
\centering
\begin{minipage}{0.32\linewidth}
\centerline{\includegraphics[width=\textwidth]{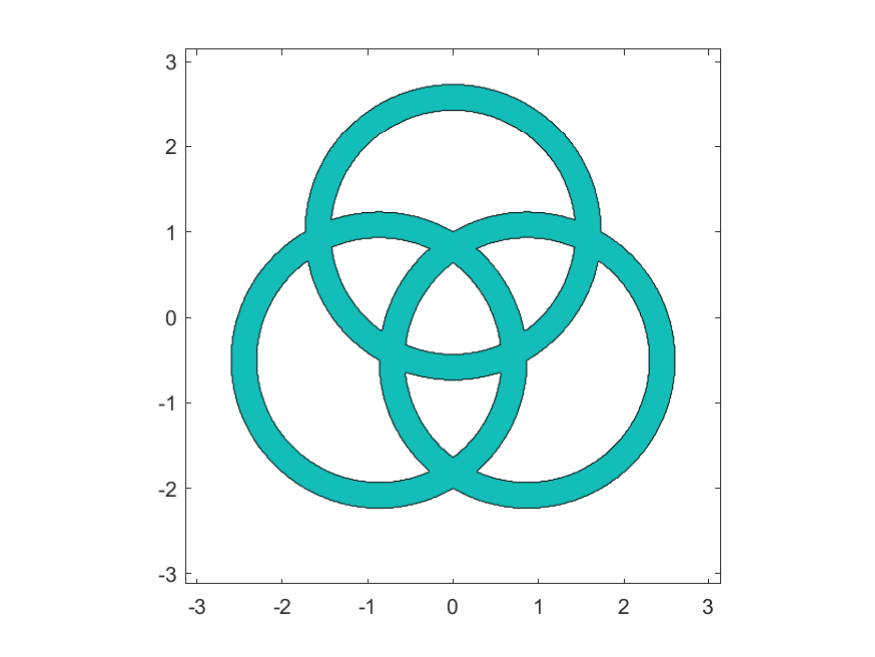}}
\centerline{t=0}
\end{minipage}
\begin{minipage}{0.32\linewidth}
\centerline{\includegraphics[width=\textwidth]{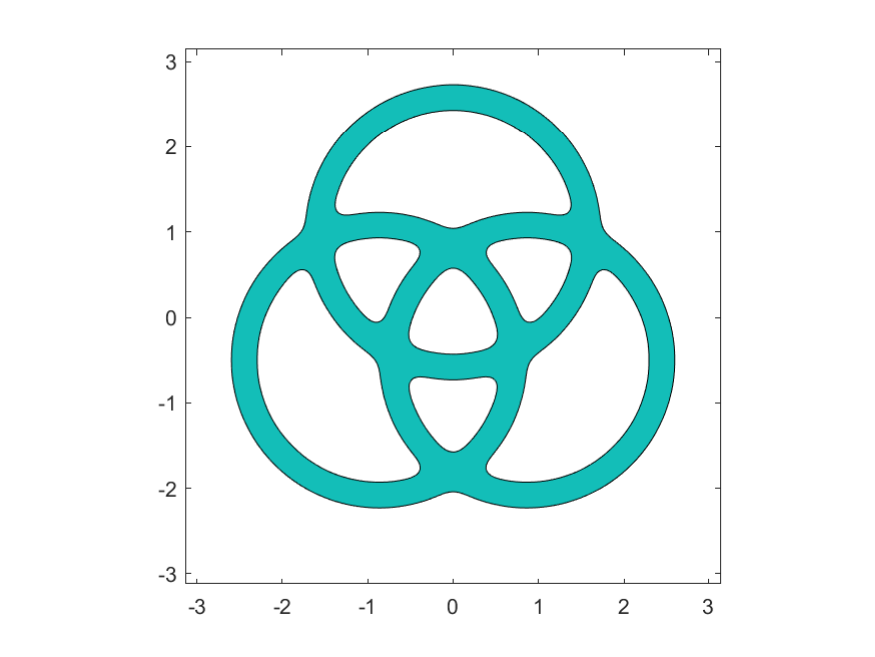}}
\centerline{t=0.005}
\end{minipage}
\begin{minipage}{0.32\linewidth}
\centerline{\includegraphics[width=\textwidth]{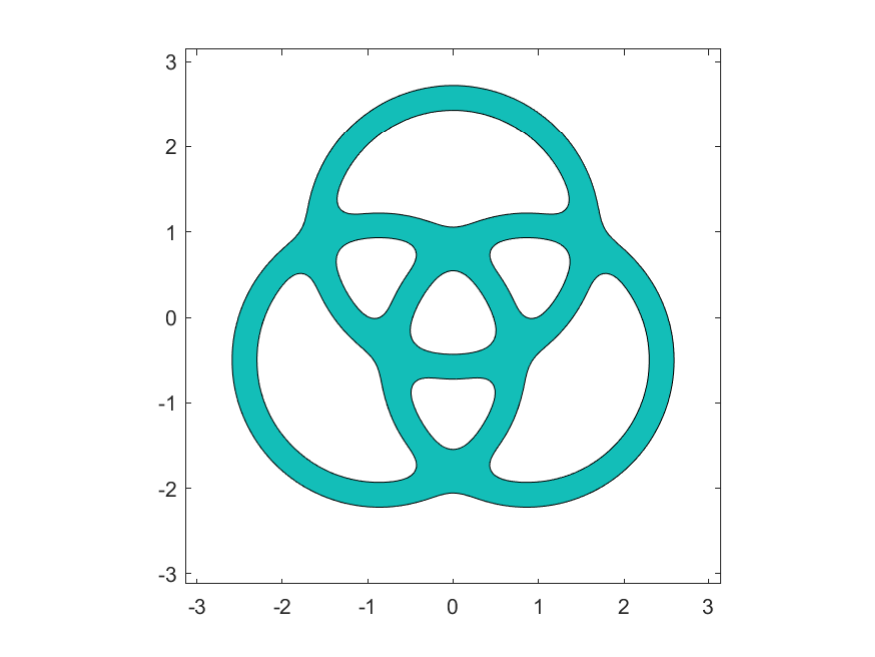}}
\centerline{t=0.01}
\end{minipage}

\begin{minipage}{0.32\linewidth}
\centerline{\includegraphics[width=\textwidth]{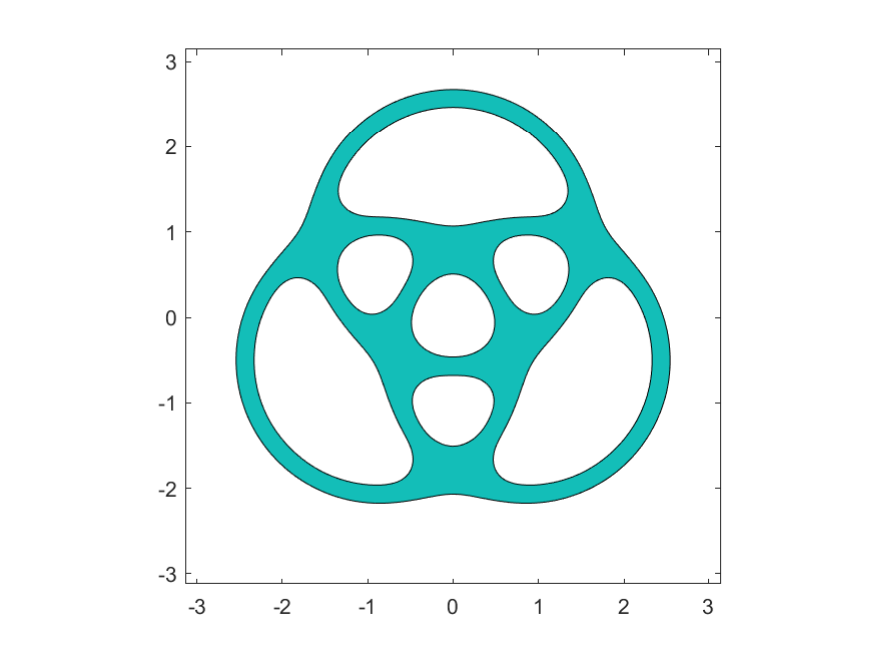}}
\centerline{t=0.02}
\end{minipage}
\begin{minipage}{0.32\linewidth}
\centerline{\includegraphics[width=\textwidth]{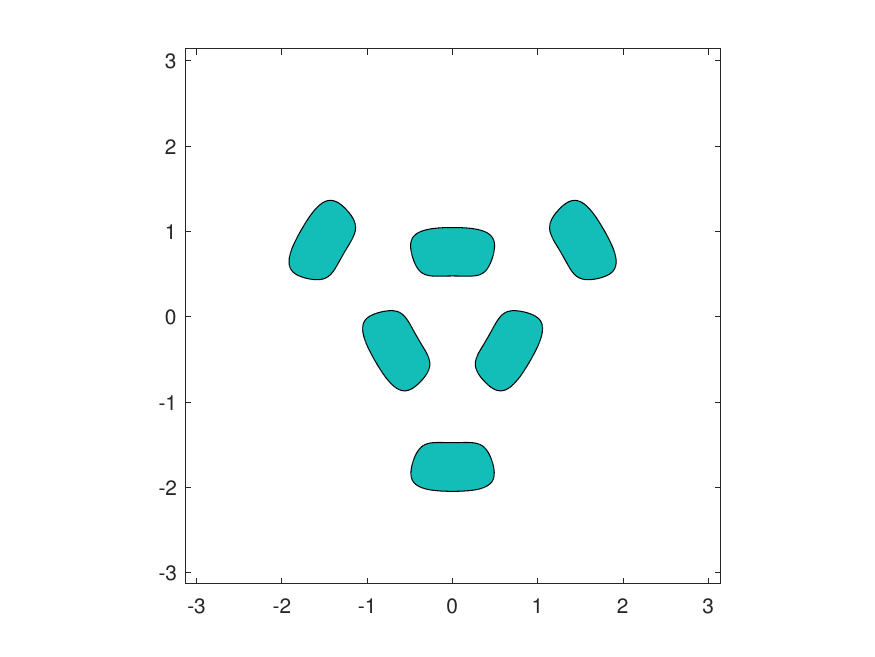}}
\centerline{t=0.05}
\end{minipage}
\begin{minipage}{0.32\linewidth}
\centerline{\includegraphics[width=\textwidth]{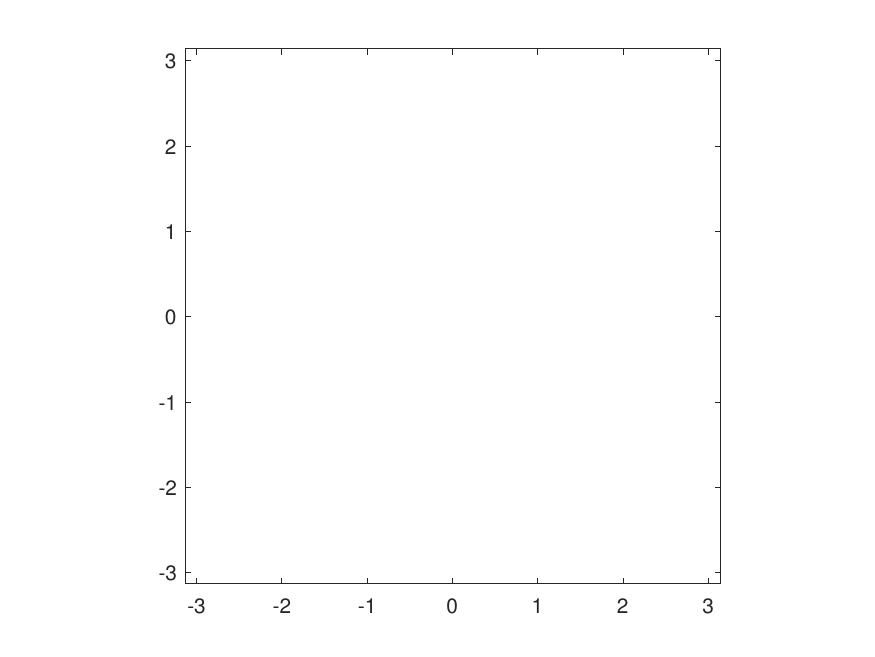}}
\centerline{t=0.1}
\end{minipage}

\caption{Solution evolution of Example 4 at $t=0$, $0.005$, $0.01$, $0.02$, $0.05$, $0.1$.}
\label{ex5_1}
\end{figure}
\section{Concluding remarks}
In this paper, we construct a series of high-order one-step schemes for parabolic equations and rigorously prove that these schemes achieve A-stability and even L-stability under specified parameter selections. Furthermore, we demonstrate that the proposed one-step schemes constitutes an effective extension of classical one-step Runge-Kutta methods. Numerical experiments conclusively validate that the new schemes attain broader stability regions than RK methods and exhibit enhanced computational efficacy across diverse parabolic problems.

\nocite{*}

\bibliographystyle{plain}
\bibliography{baer1998heat,S0021999106005651,gassner2008discontinuous,dumbser2009high,aslam1996level,
markstein2014nonsteady,rubinshteuin1971stefan,balsara2008simulating,osher2004level,sethian1999level,
ju2021maximum,martin2016extrapolated,meyer2014stabilized,hochbruck2005exponential,huang2024new,
akrivis2016backward,akrivis2015fully,li2020long,hairer1993sp,wanner1996solving,qin2023positivity}

\end{document}